\documentclass[reqno,a4paper, 11pt]{amsart}

\usepackage[a4paper=true,pdfpagelabels]{hyperref}
\usepackage{graphicx}
\usepackage{latexsym}
\usepackage{mathabx}
\usepackage{amsmath}
\usepackage{amssymb}
\usepackage[T1]{fontenc}
\usepackage[utf8]{inputenc}
\usepackage{lmodern}
\usepackage{enumitem}
\usepackage{color}   
\newtheorem{theorem}{Theorem}[section]
\newtheorem{lemma}[theorem]{Lemma}

\newtheorem{proposition}[theorem]{Proposition}

\newtheorem{lettertheorem}{Theorem}
\newtheorem{letterlemma}[lettertheorem]{Lemma}

\theoremstyle{definition}

\theoremstyle{remark}

\numberwithin{equation}{section}

\newcommand{\set}[1]{\left\{#1\right\}}
\newcommand{\abs}[1]{\left | #1\right |}
\newcommand{\nm}[1]{\left \| #1 \right \|}

\newcommand{\B}{\mathcal{B}}
\newcommand{\Bb}{\mathcal{B}^{\b,\mu}}
\newcommand{\dist}{\operatorname{dist}}
\newcommand{\D}{\mathbb{D}}
\newcommand{\DD}{\widehat{\mathcal{D}}}
\newcommand{\Dd}{\widecheck{\mathcal{D}}}

\newcommand{\DDD}{\mathcal{D}}

\newcommand{\N}{\mathbb{N}}

\newcommand{\C}{\mathbb{C}}

\newcommand{\e}{\varepsilon}

\renewcommand{\phi}{\varphi}

\newcommand{\T}{\mathbb{T}}

\def\VMOA{\mathord{\rm VMOA}}
\def\BMOA{\mathord{\rm BMOA}}

\def\a{\alpha}       \def\b{\beta}        \def\g{\gamma}
           \def\e{\varepsilon}
     \def\om{\omega}      
       \def\t{\theta}       
                  \def\z{\zeta}

\def\omg{\widehat{\omega}}

\def\mug{\widehat{\mu}}

\renewcommand{\H}{\mathcal{H}}

\newenvironment{Prf}{\noindent{\emph{Proof of}}}
{\hfill$\Box$ }

\begin{document}
	\title[Essential norms of fractional Volterra-type operators]{Essential norms of fractional Volterra-type operators}

\author[Á. M. Moreno]{\'Alvaro Miguel Moreno}
\address{Departamento de An\'alisis Matem\'atico, Universidad de M\'alaga, Campus de Teatinos, 
29071 M\'alaga, Spain}
\email{alvarommorenolopez@uma.es}
\thanks{The author is supported in part by Ministerio de Ciencia e Innovaci\'on, Spain, project PID2022-136619NB-I00; La Junta de Andaluc{\'i}a, project FQM210.}
\subjclass{26A33, 30H10, 30H35, 47G10, 47B07}
\keywords{Radial doubling weight, fractional derivative, Volterra-type operator, Essential norm, Hardy space, Bergman space}

\begin{abstract}
Given a radial doubling weight $\mu$ on the unit disk $\D$ and $f(z)=\sum_{n=0}^\infty \widehat{f}(n)z^n\in\H(\D)$, the fractional derivative of $f$ induced by $\mu$ is
    $$
    D^\mu(f)(z) = \sum_{n=0}^\infty \frac{\widehat{f}(n)}{\mu_{2n+1}}z^n,
    $$
where $\mu_{2n+1}=\int_0^1 r^{2n+1}\mu(r)\,dr$ are the odd moments of $\mu$. The fractional integral of $f$ induced by $\mu$ is $I^\mu(f)(z)=\sum_{n=0}^\infty \mu_{2n+1}\widehat{f}(n)z^n$. We study the fractional Volterra-type operator
  $$
  V_{\mu,g}(f)(z)=I^\mu(f\cdot D^\mu(g))(z),\quad f\in\H(\D),
  $$
with symbol $g\in\H(\D)$, introduced in \cite{BelMorNikPel}. First, we estimate the essential norm of $V_{\mu,g}$ acting on $H^p$ and $A^p_\a$. Later, we study the same questions for $V_{\mu,g}:H^p\to H^q$ and $V_{\mu,g}:A^p_\a\to A^q_\a$, characterizing their boundedness and compactness and estimating their essential norms for all $0<p,q<\infty$.
\end{abstract} 
\maketitle

\section{Introduction}
Let $\H(\D)$ denote the space of analytic functions in the unit disk $\D=\{z\in\C:|z|<1\}$. For a fixed function $g\in\H(\D)$ we consider the integral operator given by $$V_g(f)(z)=\int_0^z f(\z)g'(\z)\,d\z, \quad f\in \H(\D).$$

The Volterra-type operator $V_g$ has played a fundamental role in operator theory on spaces of analytic functions over the last decades. It was introduced as a proper operator by Pommerenke in the 70's \cite{Pommerenke}, and has attracted considerable attention since the seminal works of Aleman, Cima and Siskakis \cite{AlemaCima,AlemanSiskakis1,AlemanSiskakis2}. In recent years several authors showed how the study of this operator is connected with
many areas of research, including the theory of Muckenhoupt and Bekoll\'e-Bonami weights \cite{AlemanConstantin,AlemanPelaez}, H\"ormander maximal functions and Carleson measures \cite[Chapters 2 and 4]{PelRatMemoirs}, and the strong continuity of semigroups of composition operators \cite{ChalMor}. More recently, different generalizations of $V_g$ defined in terms of higher-order derivatives of the symbol $g$ have been considered \cite{ChalmoukisVolterra,DuLiQu}. In this paper, we focus on the generalization introduced in \cite{BelMorNikPel} in terms of the fractional derivative induced by a radial weight.

The extension $\mu(z)=\mu(|z|)$,\, $z\in\D$, of a non-negative function $\mu\in L^1([0,1))$  satisfying $\mug(r)=\int_{r}^1 \mu(s)\,ds>0$ for every $r\in[0,1)$, is called a radial weight. The fractional derivative of  a function $ f(z)=\sum_{n=0}^{\infty}\widehat{f}(n)z^n \in \H(\D)$  induced by a radial weight $\mu$ is given by
	$$
	D^\mu(f)(z)=\sum_{n=0}^{\infty} \frac{\widehat{f}(n)}{\mu_{2n+1}}z^n, 
	$$
where $\mu_{x}=\int_0^1 r^{x}\mu(r)\,dr$, $x\ge 0$, are the moments of $\mu$.

The fractional derivative $D^\mu$ was introduced in \cite{PelRosa} and in particular, for  $\b>0$, taking the standard weight $\mu(z)=\beta(1-|z|^2)^{\beta-1}$ recovers the fractional derivative of order $\b$ introduced by Hardy and Littlewood in \cite[p. 409]{HardyLittlewood}. We also consider the fractional integral operator $I^\mu(f)(z)=\sum_{n=0}^{\infty} \mu_{2n+1}\widehat{f}(n)z^n.$ 

Finally, the fractional Volterra type operator induced by $\mu$, with symbol $g\in\H(\D)$, is defined by
	$$
	V_{\mu,g}(f)(z)= I^\mu(f\cdot D^\mu(g))(z),\quad f\in\H(\D).
	$$
As we mentioned above, $V_{\mu,g}$ was introduced in \cite{BelMorNikPel}, where the authors studied several properties of this operator on Hardy and Bergman spaces under doubling conditions on the weight.

For $0<p\leq \infty$, the Hardy space $H^p$ consists of $f\in\H(\D)$ for which
$
    \|f\|_{H^p}=\sup_{0<r<1}M_p(r,f)<\infty$,
where
    $$
    M_p(r,f)=\left (\frac{1}{2\pi}\int_0^{2\pi}
    |f(re^{i\theta})|^p\,d\theta\right )^{\frac{1}{p}},\quad 0<p<\infty,
    $$
and 
    $$
    M_\infty(r,f)=\max_{0\le\theta\le2\pi}|f(re^{i\theta})|.
    $$ 
    
Let us recall that for $0<p<\infty$ and  a radial weight $\omega$, the Bergman  space $A^p_\om$ consists of functions $f\in \H(\D)$  such that
    $$
    \|f\|_{A^p_\omega}^p=\int_\D|f(z)|^p\omega(z)\,dA(z)<\infty,
    $$
where  $dA(z)=\frac{dx\,dy}{\pi}$ is the normalized Lebesgue area measure on $\D$. For $\a>-1$ we denote by $A^p_\a$ the standard weighted Bergman space induced by $\om_\a(z)=(\a+1)(1-|z|^2)^\a$.

In many of our results the assumption that $\mu$ is a radial doubling weight is fundamental. 
A radial weight $\mu$ belongs to $\DD$ if there exists $C=C(\mu)>0$ such that 
	$$
	\mug(r)\le C\mug\left(\frac{1+r}{2}\right),\quad 0\le r<1,
	$$
and $\mu\in\Dd$ if
	$$
	\mug(r)\ge C\mug\left(1-\frac{1-r}{K}\right),\quad 0\le r<1,
	$$ 
for some $K=K(\mu)>1$ and $C=C(\mu)>1$.
Write $\mathcal{D}=\DD\cap\Dd$, and simply say that $\mu$ is a radial doubling weight if $\mu\in\DDD$. The class $\DDD$ arises naturally in many questions in operator theory. For instance, the Bergman projection $P_\mu$ acts as a bounded and onto operator from $L^\infty$ to the Bloch space $\mathcal{B}=\{ f \in \H(\D): \|f\|_{\B}=|f(0)|+ \sup_{z\in \D}(1-|z|^2)|f'(z)|<\infty\}$ if and only if $\mu\in\DDD$. Similarly, the Littlewood-Paley formula
	\begin{equation}\label{eq: L-P estimate}
	 \int_\D|f(z)|^p \mu(z)\,dA(z)\asymp \int_\D|f'(z)|^p(1-|z|)^p\mu(z)\,dA(z)+|f(0)|^p,\quad f\in \H(\D),
	\end{equation}
holds if and only if $\mu\in\DDD$ \cite{PelaezRattyaAdvances}.  

The starting point for this paper is the following result from \cite{BelMorNikPel}. Recall that $\BMOA$ consists of $f\in\H(\D)$ such that
$$
\nm{f}_{\BMOA}^2= \abs{f(0)}^2 + \sup_{a\in\D} \frac{\int_{S(a)}\abs{f'(z)}^2 (1-|z|^2)\,dA(z)}{1-\abs{a}}<\infty.
$$
where $S(a)=\set{z\in\D: \abs{\arg{a}-\arg{z}}\leq \frac{1-\abs{a}}{2}, 1-\abs{z}\leq 1-\abs{a}}$ for $a\in \D\setminus\{0\}$, and  $S(0)=\D$, is the Carleson square induced by $a$. The spaces $\VMOA$ and $\B_0$ are the closures of the polynomials in $\BMOA$ and $\B$, respectively.
\begin{lettertheorem}\label{letterth: bella}
Let $0<p<\infty$, $\a>-1$, $\mu\in\DDD$ and $g\in \H(\D)$. Then
\begin{itemize}
    \item[(i)] $V_{\mu,g}$ is bounded on $H^p$ if and only if $g\in\BMOA$. Moreover, 
    $$
    \nm{V_{\mu,g}}_{H^p\to H^p} \asymp \nm{g}_{\BMOA};
    $$
    \item[(ii)] $V_{\mu,g}$ is compact on $H^p$ if and only if $g\in\VMOA$;
    \item[(iii)] $V_{\mu,g}$ is bounded on $A^p_\a$ if and only if $g\in\B$. Moreover, 
    $$
    \nm{V_{\mu,g}}_{A^p_\a\to A^p_\a} \asymp \nm{g}_{\B};
    $$
    \item[(iv)] $V_{\mu,g}$ is compact on $A^p_\a$ if and only if $g\in\B_0$.
\end{itemize}
\end{lettertheorem}
This result shows that $V_{\mu,g}$ and $V_g$ are simultaneously bounded (respectively, compact) on $H^p$ or $A^p_\a$ whenever $\mu\in\DDD$. Our first objective is twofold: first to obtain a quantitative version of Theorem~\ref{letterth: bella} (ii) and (iv) in terms of the essential norm, and second to compare the essential norms of $V_{\mu,g}$ and $V_g$.

Recall that, if $T:X\to Y$ is a bounded linear operator between quasi-normed spaces, the essential norm is its distance to the closed subspace of compact operators from $X$ to $Y$:
$$
\nm{T}_{e,X\to Y} = \inf\set{\nm{T-S}_{X\to Y}: S \text{ is compact}}.
$$
The essential norm of $V_g$ on various spaces of analytic functions has been extensively studied in recent years \cite{EssHardy,EssBloch,EssBergman}. In many of these results, the essential norm is compared with the distance of the symbol to the subspace of symbols inducing compactness of $V_g$.

To simplify the notation, for a quasi-normed space $X\subset \H(\D)$ and a linear subspace $X_0\subset X$, we write 
$$\dist(f,X_0) = \inf\{\nm{f-g}_X:g\in X_0\},\quad f\in X$$
for the distance to $X_0$ measured in the quasi-norm of $X$, whenever the space $X$ is clear from the context.

Our first main result is the following.
\begin{theorem}\label{th: norma esencial p=q}
Let $0<p<\infty$, $\a>-1$, $\mu\in\DDD$ and $g\in\H(\D)$. Then 
\begin{itemize}
    \item[(i)] If $V_{\mu,g}$ is bounded on $H^p$, i.e., $g\in \BMOA$, then
        $$
         \nm{V_{\mu,g}}_{e,H^p\to H^p} \asymp \dist(g,\VMOA) \asymp \limsup_{|a|\to 1^-}\left ( \frac{\int_{S(a)}\abs{g'(z)}^2 (1-|z|)dA(z)}{1-|a|}\right )^{\frac{1}{2}}.
        $$
    \item[(ii)] If $V_{\mu,g}:A^p_\a\to A^p_\a$ is bounded, i.e., $g\in \B$, then
        $$
         \nm{V_{\mu,g}}_{e,A^p_\a\to A^p_\a} \asymp \dist(g,\B_0)\asymp \limsup_{|a|\to 1^-}\abs{g'(a)}(1-|a|).
        $$    
\end{itemize}
\end{theorem}

The proof is presented in Section~\ref{Sec5} and has two main ingredients: first, a deep understanding of the space of symbols and equivalent formulas for the distance to the space of symbols, and second, the study of suitable families of test functions $f_a$ that reduce the lower estimates for the essential norm to estimates of $\limsup_{|a|\to 1^-} \nm{V_{\mu,g}(f_{a})}_{A^q_\a}$.

We highlight that this result together with \cite{EssHardy,EssBergman} shows that the essential norms of $V_{\mu,g}$ and $V_g$ acting on $H^p$ (respectively, $A^p_\a$) are comparable.

We next extend the study to any pair of indices $0<p,q<\infty$. We first characterize the symbols for which $V_{\mu,g}:H^p\to H^q$ or $V_{\mu,g}:A^p_\a\to A^q_\a$ is bounded or compact, and then estimate the essential norm in each case. To simplify the notation, we write $A^p_{-1}=H^p$ throughout the rest of the paper.

The first result covers the range $q<p$.
\begin{theorem}\label{th: caract Vmu q<p}
Let $0<q<p<\infty$, $\a\geq-1$, $\mu\in\DDD$ and $g\in\H(\D)$. Then the following conditions are equivalent:
\begin{itemize}
    \item[(i)] $V_{\mu,g}:A^p_\a\to A^q_\a$ is bounded;
    \item[(ii)] $V_{\mu,g}:A^p_\a\to A^q_\a$ is compact;
    \item[(iii)] $g\in A^s_\a$, where $\frac{1}{s}=\frac{1}{q}-\frac{1}{p}$.
\end{itemize}         
Moreover, $\nm{V_{\mu,g}}_{A^p_\a\to A^q_\a} \asymp \nm{g}_{A^s_\a}.$

\end{theorem}
The proof is presented in Section~\ref{Sec3}, where we reduce the question to an embedding problem between different spaces.

Together with Theorem~\ref{letterth: bella}, \cite[Theorem 1]{AlemaCima} and \cite[Theorem 2]{RattyaBergman}, this result shows that $V_g$ and $V_{\mu,g}$ are simultaneously bounded (respectively, compact) from $A^p_\a$ to $A^q_\a$ whenever $0<q\leq p <\infty$ and $\a\geq -1$.

The case $p<q$ is more delicate. Consider the higher-order generalization $V_{n,g}(f)(z)= I_n(f\cdot g^{(n)})(z)$ of $V_g$, where $I_0(f)=f$ and $I_n(f)(z) = \int_0^z I_{n-1}(f)(\z)\,d\z$ for $n\in\N$. The boundedness and compactness of $V_{n,g}:A^p_\a\to A^q_\a$ can be characterized in terms of the following spaces.

Let $\b\geq 0$ and $n_\b=\min\{n\in\N : n>\b\}$. We define $\B^\b$ as the space of $f\in\H(\D)$ such that
$$
\sup_{z\in\D} \abs{f^{(n_\b)}(z)}(1-|z|)^{n_\b-\b}<\infty,
$$
equipped with the norm
$$
\nm{f}_{\B^\b} = \sum_{k=0}^{n_\b-1}\abs{f^{(k)}(0)} + \sup_{z\in\D} \abs{f^{(n_\b)}(z)}(1-|z|)^{n_\b-\b},
$$
and its little subspace $\B^\b_0$ consisting of $f\in \B^\b$ satisfying
$$
\lim_{|z|\to 1^-} \abs{f^{(n_\b)}(z)}(1-|z|)^{n_\b-\b}=0.
$$
In particular, for $0<p<q<\infty$, $\a\geq -1$ and $n\in \N$ with $n>(\a+2)(\frac{1}{p}-\frac{1}{q})=\b$, the operator $V_{n,g}:A^p_\a\to A^q_\a$ is bounded (respectively, compact) if and only if $g\in\B^\b$ (respectively, $g\in \B^\b_0$); see \cite{ChalmoukisVolterra,DuLiQu}.

Motivated by this characterization, for a radial weight $\mu$ and $\b\geq0$, we define $\Bb$ as the space of $f\in\H(\D)$ such that
$$
\nm{f}_{\Bb} = \sup_{z\in\D} \abs{D^\mu(f)(z)}\mug(z)(1-|z|)^{-\b}<\infty,
$$
and $\Bb_0$ as the subspace of $f\in \Bb$ satisfying
$$
\lim_{|z|\to 1^-} \abs{D^\mu(f)(z)}\mug(z)(1-|z|)^{-\b} =0.
$$

Before stating the result for $V_{\mu,g}$,  we study the relation between $\Bb$ and $\B^\b$ to understand how $V_{\mu,g}$ compares with $V_{n,g}$. For a radial weight $\mu$, define the lower Matuszewska index by
$$
\b_\mu = \sup\set{\eta\geq0: \frac{\mug(r)}{(1-r)^\eta} \text{ is almost decreasing}}.
$$
Notice that the set defining $\b_\mu$ is non-empty because $\mug$ is decreasing. Moreover, if $\mu\in\DDD$, then $0<\b_\mu<\infty$.

\begin{theorem}\label{th: caract Bmu}
Let $\b\geq 0$ and $\mu$ be a radial weight. Then the following are equivalent:
\begin{itemize}
    \item[(i)] $\mu\in\DDD$ and $\beta < \b_\mu$;
    \item [(ii)] $\B^{\b}= \Bb$ with equivalent norms;
    \item[(iii)] $\B^{\b}\subset \Bb$, and
        $$
        \nm{f}_{\Bb} \lesssim \nm{f}_{\B^{\b}},\quad f\in\H(\D);
        $$
    \item[(iv)] $\B^{\b}_0= \Bb_0$ with equivalent norms;
    \item[(v)] $\B^{\b}_0\subset \Bb_0$, and
        $$
        \nm{f}_{\Bb} \lesssim \nm{f}_{\B^{\b}},\quad f\in\B^{\b}_0.
        $$
\end{itemize}
\end{theorem}
The case $\b=0$ was established in \cite{MorPelRosa}, while for $\b>0$ the proof follows directly from Propositions~\ref{propo: Bmu>} and \ref{propo: Bmu<}. Moreover, Proposition~\ref{propo: Bmu>} also shows that the reverse inclusions and norm inequalities corresponding to Theorem~\ref{th: caract Bmu} (iii) and (v) hold for every radial weight $\mu$.

Our final main result covers the case $p<q$.
\begin{theorem}\label{th: caract Vmu p<q}
Let $0<p<q<\infty$, $\a\geq-1$, $\mu\in\DDD$, $g\in\H(\D)$ and set $\b=(\a+2)(\frac{1}{p}-\frac{1}{q})$. Then
\begin{itemize}
    \item[(i)] $V_{\mu,g}:A^p_\a\to A^q_\a$ is bounded if and only if $g\in \Bb$, with the norm equivalence $$\nm{V_{\mu,g}}_{A^p_\a\to A^q_\a} \asymp \nm{g}_{\Bb};$$
    \item[(ii)] $V_{\mu,g}:A^p_\a\to A^q_\a$ is compact if and only if $g\in \Bb_0$, and for every $g\in\Bb$
        $$
        \nm{V_{\mu,g}}_{e,A^p_\a\to A^q_\a} \asymp \dist(g,\Bb_0)\asymp\limsup_{|z|\to 1^-}\abs{D^\mu(g)(z)}\mug(z)(1-|z|)^{-\b}.
        $$
\end{itemize}
Moreover, if $\b<\b_\mu$, then
\begin{itemize}
    \item[(iii)] $V_{\mu,g}:A^p_\a\to A^q_\a$ is bounded if and only if $g\in \B^{\b}$, with the norm equivalence $$\nm{V_{\mu,g}}_{A^p_\a\to A^q_\a} \asymp \nm{g}_{\B^{\b}};$$
    \item[(iv)] $V_{\mu,g}:A^p_\a\to A^q_\a$ is compact if and only if $g\in \B^{\b}_0$, and for every $g\in \B^\b$
        $$
        \nm{V_{\mu,g}}_{e,A^p_\a\to A^q_\a} \asymp \dist(g,\B^{\b}_0)\asymp\limsup_{|z|\to 1^-}\abs{g^{(n_\b)}(z)}(1-|z|)^{n_\b-\b}.
        $$
\end{itemize}
\end{theorem}

Together with Proposition 4.2, this theorem shows that, for $0<p<q<\infty$, $\a>-1$, $\mu\in\DDD$ and every $n>(\a+2)(\frac{1}{p}-\frac{1}{q})=\b$, boundedness (respectively, compactness) of $V_{\mu,g}:A_\alpha^p\to A_\alpha^q$ implies boundedness (respectively, compactness) of $V_{n,g}:A_\alpha^p\to A_\alpha^q$. In addition, if $\b<\beta_\mu$, the two operators are simultaneously bounded (respectively, compact).

The rest of the paper is organized as follows. Section~\ref{Sec2} collects some technical results on radial weights and useful representation formulas for the fractional derivative. Section~\ref{Sec3} studies the boundedness and compactness of $V_{\mu,g}:A^p_\a\to A^q_\a$ for $\a\geq -1$, concluding with the proof of Theorem~\ref{th: caract Vmu q<p}. Theorem~\ref{th: caract Bmu} is proved in Section~\ref{Sec4}. Finally, Section~\ref{Sec5} develops the essential norm estimates and all the necessary results leading to the proof of Theorem~\ref{th: norma esencial p=q} and Theorem~\ref{th: caract Vmu p<q}.

The letter $C=C(\cdot)$ will denote an absolute constant whose value depends on the parameters indicated in parenthesis, and may change from one occurrence to another.
We will use the notation $a\lesssim b$ if there exists a constant
$C=C(\cdot)>0$ such that $a\le Cb$, and $a\gtrsim b$ is understood
in an analogous manner. In particular, if $a\lesssim b$ and
$a\gtrsim b$, then we write $a\asymp b$ and say that $a$ and $b$ are comparable.

\section{Radial weights and technical lemmas}\label{Sec2}
Throughout the manuscript we will employ different descriptions of the classes of radial weights $\DD$ and $\Dd$. The next result gathers several characterizations of $\DD$ proved in \cite[Lemma 2.1]{PelaezSmallWeight}. 
\begin{letterlemma}\label{lemma: caract Dgorro}
Let $\om$ be a radial weight. Then the following statements are equivalent:
\begin{itemize}
	\item[(i)] $\om\in\DD$;
	\item[(ii)] There exists $\gamma>0$ and a constant $C=C(\gamma,\om)>0$ such that
        $$\omg(t)\leq C \left ( \frac{1-t}{1-s} \right )^\gamma \omg(s),\quad 0\leq t \leq s <1;$$
	\item[(iii)] There exists a constant $C=C(\om)>0$ such that
		$$
		\om_x\leq C\om_{2x},\quad x\geq1;
        $$
    \item[(iv)] There exists a constant $C=C(\om)>0$ such that
        $$
        \om_x \leq C \omg \left (1-\frac{1}{x} \right ),\quad 1\leq x<\infty.
        $$
\end{itemize}
\end{letterlemma}
We will also use a couple of descriptions of the class $\Dd$, see \cite[Lemma B]{PelRosa}.
\begin{letterlemma}\label{lemma: caract dcheck}
Let $\om$ be a radial weight. Then the following statements are equivalent: 
\begin{itemize}
	\item[(i)] $\om\in\Dd$;
	\item[(ii)] There exists $\eta>0$ and a constant $C=C(\eta,\om)>0$ such that
	\begin{equation*}
		\omg(s)\leq C \left ( \frac{1-s}{1-t} \right )^\eta \omg(t),\quad 0\leq t \leq s <1;
	\end{equation*}
\item[(iii)] For each (equivalently for some) $\gamma>0$ there is $\eta_0=\eta_0(\gamma)< 1$ such that for any  $\eta_0<\eta\leq 1$ there is $C=C(\om,\gamma, \eta)>0$ such that
	$$
	\int_0^r \frac{ds}{\omg(s)^\gamma(1-s)^\eta} \leq \frac{C}{\omg(r)^\gamma(1-r)^{\eta-1}},\quad 0\leq r < 1.
	$$  
\end{itemize}
\end{letterlemma}

Finally, we prove a technical lemma. We write $D(z,r)$ for the Euclidean disk with center $z$ and radius $r$.
\begin{lemma}\label{lemma: esp crecimiento}
Let $\a\geq -1$, $0<\beta,\g<\infty$, $\mu\in\DD$ and $G\in\H(\D)$. Then
\begin{itemize}
    \item[(a)] The following conditions are equivalent:
    \begin{itemize}
        \item[(ai)] $$
        \sup_{z\in\D} \abs{G(z)}\mug(z)(1-\abs{z})^{-\b(\a+2)}<\infty;
        $$
        \item[(aii)]
        $$\sup_{a\in\D} \frac{\int_{S(a)} \abs{G(z)}^\g\mug(z)^\g(1-\abs{z})^\a dA(z)}{(1-|a|)^{(\a+2)(1+\g\b)}}<\infty.$$
    \end{itemize}
Moreover,
$$
\left (\sup_{z\in\D} \abs{G(z)}\mug(z)(1-\abs{z})^{-\b(\a+2)} \right )^\g \asymp \sup_{a\in\D} \frac{\int_{S(a)} \abs{G(z)}^\g\mug(z)^\g(1-\abs{z})^\a dA(z)}{(1-|a|)^{(\a+2)(1+\g\b)}},\quad G\in\H(\D)
$$
\item[(b)] The following conditions are equivalent:
\begin{itemize}
    \item[(bi)] $$
        \lim_{|z|\to 1^-} \abs{G(z)}\mug(z)(1-\abs{z})^{-\b(\a+2)}=0;
            $$
    \item[(bii)]
        $$\lim_{|a|\to 1^-} \frac{\int_{S(a)} \abs{G(z)}^\g\mug(z)^\g(1-\abs{z})^\a dA(z)}{(1-|a|)^{(\a+2)(1+\g\b)}}=0.$$
\end{itemize}
\end{itemize}
\end{lemma}
\begin{proof}
   We first prove (a). Assume that $(ai)$ holds, that is,
    $$
    \rho(G)=\sup_{z\in\D} \abs{G(z)}\mug(z)(1-\abs{z})^{-\b(\a+2)}<\infty,
    $$
    then, for each $a\in\D$,
    \begin{equation}\label{eq:lemma esp crecimiento 1}
        \begin{split}
            \int_{S(a)} \abs{G(z)}^\g\mug(z)^\g(1-\abs{z})^\a dA(z) &\leq \rho(G)^\g \int_{S(a)} (1-\abs{z})^{\a+\g\b(\a+2)} dA(z) \\
            &\lesssim \rho(G)^\g (1-|a|)^{(\a+2)(1+\g\b)},
        \end{split}
    \end{equation}
    so $(aii)$ follows by taking the supremum over $a\in\D$.
    
    Conversely, assume that $(aii)$ holds. For $|z|>\frac{1}{2}$, set $z^\star = \frac{3|z|-1}{2}e^{i\arg{z}}$, then by subharmonicity and the fact that $\mu\in\DD$, 
    \begin{equation}\label{eq:lemma esp crecimiento 2}
    \begin{split}
    \abs{G(z)}^\g\mug(z)^\g(1-\abs{z})^{-\g\b(\a+2)} &\lesssim \frac{(1-|z|)^{-\g\b(\a+2)}}{(1-|z|)^{\a+2}}\int_{D(z,\frac{1-|z|}{4})} \abs{G(\z)}^\g\mug(\z)^\g(1-|\z|)^\a dA(\z)\\
    &\lesssim \frac{\int_{S(z^\star)} \abs{G(\z)}^\g\mug(\z)^\g(1-\abs{\z})^\a dA(\z)}{(1-|z^\star|)^{(\a+2)(1+\g\b)}}\\
    &\leq \sup_{a\in\D} \frac{\int_{S(a)} \abs{G(z)}^\g\mug(z)^\g(1-\abs{z})^\a dA(z)}{(1-|a|)^{(\a+2)(1+\g\b)}}.
    \end{split}
    \end{equation}
    Concluding the proof of $(a)$.

    The proof of $(b)$ requires only minor modifications in the proof of $(a)$. Assume first that $(bi)$ holds. Given $\e>0$, there exists $0<r<1$ such that
    $$
    \abs{G(z)}\mug(z)(1-\abs{z})^{-\b(\a+2)} < \e^{\frac{1}{\g}},\quad r\leq|z|<1,
    $$
    then by reproducing the computation in \eqref{eq:lemma esp crecimiento 1} gives
    $$
    \frac{\int_{S(a)} \abs{G(z)}^\g\mug(z)^\g(1-\abs{z})^\a dA(z)}{(1-|a|)^{(\a+2)(1+\g\b)}} \lesssim \e,\quad r<|a|<1,
    $$
    which proves $(bii)$. Conversely, assume that $(bii)$ holds. Given $\e>0$, choose $0<r<1$ such that 
    $$
    \frac{\int_{S(a)} \abs{G(z)}^\g\mug(z)^\g(1-\abs{z})^\a dA(z)}{(1-|a|)^{(\a+2)(1+\g\b)}} < \e^\g,\quad r<|a|<1,
    $$
    If $\frac{1+2r}{3}<|z|<1$, then $r<|z^\star|<1$, so \eqref{eq:lemma esp crecimiento 2} yields
    $$
    \abs{G(z)}\mug(z)(1-\abs{z})^{-\b(\a+2)} \lesssim \e,\quad \frac{1+2r}{3}<|z|<1,
    $$
    concluding the proof.
\end{proof}
    \subsection{Reproducing formulas and useful identities}\label{subse:22}
The   Bergman reproducing kernel  induced by a radial weight $\om$, can be written as 
\begin{equation}\label{eq:B}
B^\om_z(\z)=\sum_{n=0}^\infty\frac{\left(\overline{z}\z\right)^n}{2\om_{2n+1}}, \quad z,\z\in \D.
\end{equation}

We next consider a fractional integral-derivative operator that includes $D^\mu$ as a particular case. For a pair of radial weights $\om$ and $\nu$, the fractional integral-derivative operator
	$$
	R^{\om,\nu}(f)(z)=\sum_{n=0}^{\infty}\frac{\om_{2n+1}}{\nu_{2n+1}}\widehat{f}(n)z^n,\quad f\in\H(\D),
	$$
was first introduced in \cite{PeralaFractional}, and a calculation shows (see \cite{PeralaFractional,PeralaRattyaWang}) that
\begin{equation}\label{eq:mu+}
   D^\mu(f)(z) = R^{1,\mu_+}(f)(z),\quad f\in\H(\D),\quad z\in\D,
\end{equation}
where  $\mu_+(z)=\int_{\abs{z}}^1\frac{\mu(s)}{s}ds,$ for $z\in\D\setminus\{0\}$. We recall the following integral representation of $R^{\om,\nu}$ from \cite[(3.3)]{PeralaFractional}.
	\begin{lettertheorem}\label{th:perala}
		Let $\om$ and $\nu$ be radial weights. Then
		$$
		R^{\om,\nu}(f)(z) = \int_{\D} f(\z)\overline{B^{\nu}_{z}(\z)}\om(\z)dA(\z),\quad f\in A^1_\om,\quad z\in\D.
		$$
	\end{lettertheorem}
This integral representation yields an alternative formula for $D^{\mu}$ in terms of the $n$th derivative. To state it, we introduce the following weights. For a radial weight $\mu$ we consider the iterated radial weights 
	$$
	V_{\mu_+,n}(z)= 2\int_{\abs{z}}^1 rV_{\mu_+,n-1}(r)dr,\quad z\in \D,\, n\in\N\quad \text{and} \quad V_{\mu_+,0}(z)=\mu_+(z),
	$$
	and for a radial weight $\om$ we consider
	$\om^\ast(z)=\int_{\abs{z}}^1s\om(s)\log\frac{s}{\abs{z}}ds$ for $z\in\D\setminus\{0\}$ and its iterations
	$$
	\om^{\ast,n}(z)=\int_{\abs{z}}^1s\om^{\ast,n-1}(s)\log\frac{s}{\abs{z}}ds,\quad z\in \D\setminus\{0\},n\in \N \quad \text{and} \quad \om^{\ast,0}(z)=\om(z).
	$$ 
In the particular case $\om=1$, we will write  
$W^n = \om^{\ast,n}, \, n\in \N\cup\{0\}.$
The following formula relates $D^\mu(f)$ with the $n$th derivative of $f$.

	\begin{letterlemma}\cite[Corollary 2.2]{BelMorNikPel}\label{lemmaletter: expresión derivada fraccionaria}
		Let $\mu$ be a radial weight. Then 
		$$
		D^{\mu}(f)(z) = \sum_{j=0}^{n-1} \frac{\widehat{f}(j)}{\mu_{2j+1}}z^j + 4^n z^n R^{W^n,V_{\mu_+,n}}(f^{(n)})(z),\quad f\in A^1,z\in\D,n\in \N.
		$$
	\end{letterlemma}
	The following is a technical lemma that will be useful for our purposes.

	\begin{letterlemma}\cite[Lemma 2.3]{BelMorNikPel}\label{lema pesos tec}
		Let $\mu$ be a radial weight. Then
		\begin{itemize}
			\item[(i)] If $n\in\N$, then 
			\begin{equation}\label{eq: lema pesos cond 1}
				W^n(r) \asymp (1-r)^{2n},\quad \frac{1}{2}\leq r<1;
			\end{equation}
			\item[(ii)]
			$$
			\widehat{\mu_+}(r) \leq\mug(r) (1-r),\quad 0\leq r <1;
			$$
			\item[(iii)] If $n\in\N$, then
			\begin{equation}\label{eq: lema pesos cond 3}
			V_{\mu_+,n}(r) \lesssim \mug(r)(1-r)^n,\quad 0\leq r <1;
			\end{equation}
			\item[(iv)] If $\mu\in\DD$, then 
			$$
			\widehat{\mu_+}(r) \gtrsim \mug(r)(1-r),\quad 0\leq r <1;
			$$
			\item[(v)] If $\mu\in \DD$ and $n\in\N$, then
			\begin{equation}\label{eq: lema pesos cond 5}
				\widehat{V_{\mu_+,n}}(r) \gtrsim \mug(r)(1-r)^{n+1},\quad \frac{1}{2}\leq r <1.
			\end{equation}
		\end{itemize}
	\end{letterlemma}
\section{Boundedness and compactness}\label{Sec3}
We first characterize the boundedness and compactness of $V_{\mu,g}:A^p_\a\to A^q_\a$ for $\a\geq -1$. The following lemma reduces these questions to known embedding problems. Recall that, for $0<p,q<\infty$ and a positive Borel measure $\nu$ on $\D$, the tent space $T_p^q(\nu)$ consists of complex-valued measurable functions $f$ on $\D$ such that 
	$$
	\|f\|_{T_{p}^{q}(\nu)}
	=  \left(\int_{\T}\left(\int_{\Gamma(\z)}|f(z)|^p \frac{d\nu(z)}{1-|z|}\right)^{\frac{q}{p}}\,|d\z|\right)^{\frac1q}<\infty,
    $$
where $\Gamma(\z)=\left\{z\in\D:|\arg z-\arg\z|<1-|z|\right\}$ is the non-tangential approach region with vertex $\z\in\T=\partial\D$.

\begin{lemma}\label{Lemma: Vmu to embed}
        Let $0<p,q<\infty$, $\a>-1$, $\mu\in\DDD$ and $g\in\H(\D)$. Define the measures
        $$d\mu_{g,q,\a}(z) = \abs{D^\mu(g)(z)}^q\mug(z)^q(1-|z|)^\a dA(z)$$
        and $d\mu_{g}(z)=\abs{D^\mu(g)(z)}^2\frac{\mug(z)^2}{1-|z|}dA(z)$. Then 
        \begin{itemize}
            \item[(i)] $V_{\mu,g}:H^p\to H^q$ is bounded (respectively, compact) if and only if $Id:H^p\to T_{2}^{q}(\mu_{g})$ is bounded (respectively, compact). Moreover, $$\nm{V_{\mu,g}}_{H^p\to H^q}\asymp \nm{Id}_{H^p\to T_{2}^{q}(\mu_{g})};$$
            \item[(ii)] $V_{\mu,g}:A^p_\a\to A^q_\a$ is bounded (respectively, compact) if and only if $Id:A^p_\a\to L^q(\mu_{g,q,\a})$ is bounded (respectively, compact). Moreover, $$\nm{V_{\mu,g}}_{A^p_\a\to A^q_\a}\asymp \nm{Id}_{A^p_\a\to L^{q}(\mu_{g,q,\a})}.$$
            
        \end{itemize}
    \end{lemma}
    \begin{proof}
    $(i)$ follows by an application of \cite[Theorem 1.1]{BelMorNikPel}, showing that for every $\mu\in\DDD$ and $f,g\in\H(\D)$
    $$
    \nm{V_{\mu,g}(f)}_{H^q}^q \asymp \int_{\T}\left(\int_{\Gamma(\z)}|f(z)|^2 \abs{D^\mu(g)(z)}^2\frac{\mug(z)^2}{(1-|z|)^2}dA(z)\right)^{\frac{q}{2}}\,|d\z| = \nm{f}^q_{T_{2}^{q}(\mu_{g})},
    $$
    Thus, boundedness or compactness of $V_{\mu,g}: H^p\to H^q$ is equivalent to the corresponding properties of the identity $Id:H^p\to T_{2}^{q}(\mu_{g})$.
    
    The proof of $(ii)$ is similar. By \cite[Theorem 1]{PelRosa}, for every $\mu\in\DDD$ and $f,g\in\H(\D)$
    $$
    \nm{V_{\mu,g}(f)}_{A^q_\a}^q \asymp \int_\D \abs{f(z)}^q \abs{D^\mu(g)(z)}^q\mug(z)^q (1-|z|)^\a dA(z) = \nm{f}_{L^q(\mu_{g,q,\a})}^q.
    $$
    \end{proof}
    
    These embeddings have been studied by several authors; see \cite[Theorems 3 and 9]{PelRatSierraEmbedding} and \cite[Theorems 1.1 and 1.2]{LvPau}. We summarize the relevant results below. We write $\Delta(z,r)=\{\z\in\D : \abs{\frac{z-\z}{1-\overline{z}\z}}<r \}$ for the pseudohyperbolic disk with center $z\in\D$ and radius $0< r <1$.
\begin{lettertheorem}\label{thletter: emebedding bergman}
    Let $0<p,q<\infty$, $\a>-1$, and let $\nu$ be a positive Borel measure on $\D$ that is finite on compact subsets. Then
    \begin{itemize}
        \item[(a)] If $p\leq q$, then $Id:A^p_\a\to L^q_\nu$ is bounded if and only if $\sup_{a\in\D}\frac{\nu(S(a))}{(1-|a|)^{(\a+2)\frac{q}{p}}}<\infty$. Moreover,
        $$
        \nm{Id}_{A^p_\a\to L^q_\nu}^q \asymp \sup_{a\in\D}\frac{\nu(S(a))}{(1-|a|)^{(\a+2)\frac{q}{p}}}.
        $$
        \item[(b)] If $p\leq q$, then $Id:A^p_\a\to L^q_\nu$ is compact if and only if $\lim_{|a|\to 1^-}\frac{\nu(S(a))}{(1-|a|)^{(\a+2)\frac{q}{p}}}=0$. 
        \item[(c)] If $q<p$, then the following conditions are equivalent:
        \begin{itemize}
            \item[(ci)] $Id:A^p_\a\to L^q_\nu$ is bounded;
            \item[(cii)] $Id:A^p_\a\to L^q_\nu$ is compact;
            \item[(ciii)] There exists an $r\in(0,1)$ such that
            $$
            \int_\D \left (\frac{\nu(\Delta(z,r))}{(1-|z|)^{\a+2}} \right )^{\frac{p}{p-q}}(1-|z|)^\a dA(z) <\infty.
            $$
        \end{itemize}
        Moreover,
        $$
        \nm{Id}_{A^p_\a\to L^q_\nu}^{\frac{pq}{p-q}}\asymp \int_\D \left (\frac{\nu(\Delta(z,r))}{(1-|z|)^{\a+2}} \right )^{\frac{p}{p-q}}(1-|z|)^\a dA(z)
        $$
    \end{itemize}
\end{lettertheorem}
\begin{lettertheorem}\label{thletter: emebedding hardy tent}
    Let $0<p,q<\infty$, and let $\nu$ be a positive Borel measure on $\D$ that is finite on compact sets and satisfies $\nu(\{0\})=0$. Then
    \begin{itemize}
        \item[(a)] If $p\leq q$, then $Id:H^p\to T_{2}^{q}(\nu)$ is bounded if and only if $\sup_{a\in\D}\frac{\nu(S(a))}{(1-|a|)^{1+2(\frac{1}{p}-\frac{1}{q})}}<\infty$. Moreover,
        $$
        \nm{Id}_{H^p\to T_{2}^{q}(\nu)}^2 \asymp \sup_{a\in\D}\frac{\nu(S(a))}{(1-|a|)^{1+2(\frac{1}{p}-\frac{1}{q})}}.
        $$
        \item[(b)]If $p\leq q$, then $Id:H^p\to T_{2}^{q}(\nu)$ is compact if and only if $\lim_{|a|\to 1^-}\frac{\nu(S(a))}{(1-|a|)^{1+2(\frac{1}{p}-\frac{1}{q})}}=0$.  
        \item[(c)] If $q<p$, then the following conditions are equivalent:
        \begin{itemize}
            \item[(ci)] $Id:H^p\to T_{2}^{q}(\nu)$ is bounded;
            \item[(cii)] $Id:H^p\to T_{2}^{q}(\nu)$ is compact;
            \item[(ciii)] The following condition holds:
            $$
            B_\nu(\z) := \int_{\Gamma(\z)}\frac{d\nu(z)}{1-|z|} \in L^{\frac{qp}{2(p-q)}}(\T).
            $$
        \end{itemize}
        Moreover,
        $$
        \nm{Id}_{H^p\to T_{2}^{q}(\nu)}^{2}\asymp \nm{B_\nu}_{L^{\frac{qp}{2(p-q)}}(\T)}.
        $$
    \end{itemize}
\end{lettertheorem}
Combining these results with Lemma~\ref{Lemma: Vmu to embed} gives a first characterization of boundedness and compactness of $V_{\mu,g}$ when $p<q$, while for $q<p$ we prove Theorem~\ref{th: caract Vmu q<p} directly.

First, for the case $p<q$, we obtain the following proposition.
\begin{proposition}\label{propo: caract Vmu p<q}
     Let $0<p<q<\infty$, $\a\geq-1$, $\mu\in\DDD$ and $g\in\H(\D)$. Then 
    \begin{itemize}
        \item[(i)] $V_{\mu,g}:A^p_\a\to A^q_\a$ is bounded if and only if 
        $$
        \sup_{z\in\D}\abs{D^\mu(g)(z)}\mug(z)(1-|z|)^{(\a+2)(\frac{1}{q}-\frac{1}{p})}<\infty.
        $$
        Moreover,
        $$
        \nm{V_{\mu,g}}_{A^p_\a\to A^q_\a} \asymp \sup_{z\in\D}\abs{D^\mu(g)(z)}\mug(z)(1-|z|)^{(\a+2)(\frac{1}{q}-\frac{1}{p})};
        $$
        \item[(ii)] $V_{\mu,g}:A^p_\a\to A^q_\a$ is compact if and only if 
        $$
        \lim_{|z|\to 1^-}\abs{D^\mu(g)(z)}\mug(z)(1-|z|)^{(\a+2)(\frac{1}{q}-\frac{1}{p})}=0.
        $$
    \end{itemize}
\end{proposition}
\begin{proof}
We prove $(i)$; the proof of $(ii)$ is similar. We treat the Hardy case and the Bergman cases separately. First, in the Hardy space, Lemma~\ref{Lemma: Vmu to embed} (i) and Theorem~\ref{thletter: emebedding hardy tent} (a) show that $V_{\mu,g}:H^p\to H^q$ is bounded if and only if
$$
\sup_{a\in\D} \frac{\int_{S(a)} \abs{D^\mu(g)(z)}^2\frac{\mug(z)^2}{1-|z|} dA(z)}{(1-|a|)^{1+2(\frac{1}{p}-\frac{1}{q})}}<\infty,
$$
and
$$
\nm{V_{\mu,g}}_{H^p\to H^q}^2 \asymp \sup_{a\in\D} \frac{\int_{S(a)} \abs{D^\mu(g)(z)}^2\frac{\mug(z)^2}{1-|z|} dA(z)}{(1-|a|)^{1+2(\frac{1}{p}-\frac{1}{q})}}.
$$
Now an application of Lemma~\ref{lemma: esp crecimiento} (a) with $G=D^\mu(g)$, $\a=-1$, $\b=\frac{1}{p}-\frac{1}{q}$ and $\gamma=2$ gives the desired result.

On the other hand, if $\a>-1$, by Lemma~\ref{Lemma: Vmu to embed} (ii) and Theorem~\ref{thletter: emebedding bergman} (a), 
$V_{\mu,g}:A^p_\a\to A^q_\a$ is bounded if and only if
$$
\sup_{a\in\D} \frac{\int_{S(a)}\abs{D^\mu(g)(z)}^q\mug(z)^q(1-|z|)^\a dA(z)}{(1-|a|)^{(\a+2)\frac{q}{p}}}<\infty,
$$
and
$$
\nm{V_{\mu,g}}_{A^p_\a \to A^q_\a}^q \asymp \sup_{a\in\D} \frac{\int_{S(a)}\abs{D^\mu(g)(z)}^q\mug(z)^q(1-|z|)^\a dA(z)}{(1-|a|)^{(\a+2)\frac{q}{p}}},
$$
then applying Lemma~\ref{lemma: esp crecimiento} (a) with $G=D^\mu(g)$, $\beta=\frac{1}{p}-\frac{1}{q}$ and $\gamma=q$ gives
$$
\nm{V_{\mu,g}}_{A^p_\a \to A^q_\a} \asymp \sup_{z\in\D} \abs{D^\mu(g)(z)}\mug(z)(1-|z|)^{(\a+2)(\frac{1}{q}-\frac{1}{p})}<\infty,
$$
so $(i)$ holds. The proof of $(ii)$ follows the same argument, using Theorem~\ref{thletter: emebedding hardy tent} (b) in the Hardy case or Theorem~\ref{thletter: emebedding bergman} (b) in the Bergman case, followed by Lemma~\ref{lemma: esp crecimiento} (b).
\end{proof}

To conclude this section we proceed to prove Theorem~\ref{th: caract Vmu q<p}.

\begin{Prf}{\em{  Theorem~\ref{th: caract Vmu q<p}}.}
We treat the Hardy and Bergman case separately. In the Hardy case, Lemma~\ref{Lemma: Vmu to embed} and Theorem~\ref{thletter: emebedding hardy tent} $(c)$ shows that $(i)$, $(ii)$ and the condition 
    $$
    \int_\T \left (\int_{\Gamma(\z)}\abs{D^\mu(g)(z)}^2\frac{\mug(z)^2}{(1-|z|)^2}dA(z) \right )^{\frac{qp}{2(p-q)}} |d\z|<\infty,
    $$
    are equivalent.
Now, since $s=\frac{pq}{p-q}$, \cite[Theorem 1.1]{BelMorNikPel} shows that this condition is equivalent to $g\in H^s$, and 
    $$
    \nm{V_{\mu,g}}_{H^p\to H^q} \asymp \left (\int_\T \left (\int_{\Gamma(\z)}\abs{D^\mu(g)(z)}^2\frac{\mug(z)^2}{(1-|z|)^2}dA(z) \right )^{\frac{s}{2}} |d\z|\right )^{\frac{1}{s}} \asymp \nm{g}_{H^s},
    $$
concluding the proof in the Hardy case.

If $\a>-1$, the equivalence between $(i)$ and $(ii)$ follows from Lemma~\ref{Lemma: Vmu to embed} and Theorem~\ref{thletter: emebedding bergman} (c). To prove that $(iii)$ implies $(i)$, assume that $g\in A^s_\a$. By \cite[Theorem 1]{PelRosa} and Hölder's inequality, for every $f\in A^p_\a$,
    \begin{equation*}
    \begin{split}
        \nm{V_{\mu,g}(f)}_{A^q_\a}^q &\asymp \int_\D \abs{f(z)}^q \abs{D^\mu(g)(z)}^q \mug(z)^q (1-|z|)^\a dA(z) \\
        &\leq \nm{f}_{A^p_\a}^q \left (\int_\D\abs{D^\mu(g)(z)}^s \mug(z)^s (1-|z|)^\a dA(z) \right )^{\frac{q}{s}} \asymp\nm{f}_{A^p_\a}^q \nm{g}_{A^s_\a}^q,
    \end{split}
    \end{equation*}
This proves $(i)$ holds and the estimate
    $$
     \nm{V_{\mu,g}}_{A^p_\a\to A^q_\a} \lesssim \nm{g}_{A^s_\a}.
    $$
    It remains to prove that $(i)$ implies $(iii)$. By  Lemma~\ref{Lemma: Vmu to embed} and Theorem~\ref{thletter: emebedding bergman} (c), condition $(i)$ implies that there exists $r\in(0,1)$ such that
    $$
    \int_\D \left (\frac{\int_{\Delta(z,r)}\abs{D^\mu(g)(\z)}^q\mug(\z)^q(1-|\z|)^\a dA(\z)}{(1-|z|)^{\a+2}} \right )^{\frac{p}{p-q}}(1-|z|)^\a dA(z)<\infty.
    $$
    Now, by \cite[Theorem 1]{PelRosa}, subharmonicity, and the identity $s=\frac{pq}{p-q}$,
    \begin{equation*}
        \begin{split}
            \nm{g}_{A^s_\a}^s &\asymp \int_\D \abs{D^\mu(g)(z)}^s \mug(z)^s (1-|z|)^\a dA(z) \\
            &\lesssim \int_\D \left (\frac{\int_{\Delta(z,r)}\abs{D^\mu(g)(\z)}^q\mug(\z)^q(1-|\z|)^\a dA(\z)}{(1-|z|)^{\a+2}} \right )^{\frac{p}{p-q}}(1-|z|)^\a dA(z)<\infty.
        \end{split}
    \end{equation*}
    Thus $g\in A^s_\a$, and
    $$
     \nm{g}_{A^s_\a} \lesssim \nm{V_{\mu,g}}_{A^p_\a\to A^q_\a},
    $$
    concluding the proof.
\end{Prf}
\section{Characterization of the Besov-Bloch space}\label{Sec4}
This section is devoted to the proof of Theorem~\ref{th: caract Bmu}. We first record a technical lemma that follows directly from Cauchy's formula.
\begin{lemma}\label{lemma: control f^n(0)}
    Let $f\in\H(\D)$, $n\in\N\cup \{0\}$, $\b>0$, and let $\mu$ be a radial weight. Then
    \begin{itemize}
        \item[(i)] There exists a constant $C(n,\b,\mu)>0$ such that
        $$
        \abs{f^{(n)}(0)} \leq C \nm{f}_{\Bb};
        $$
        \item[(ii)] There exists a constant $C(n,\b)>0$ such that
        $$
        \abs{f^{(n)}(0)} \leq C \nm{f}_{\B^{\b}}.
        $$
    \end{itemize}
\end{lemma}
We establish the two embeddings between $\Bb$ and $\B^\b$ for $\b>0$ in separate propositions. The first holds for any radial weight.
\begin{proposition}\label{propo: Bmu>}
    Let $\beta>0$ and $\mu$ be a radial weight. Then
    \begin{itemize}
        \item[(i)] $\Bb\subset \B^{\b}$, and
        $$
        \nm{f}_{\B^{\b}} \lesssim \nm{f}_{\Bb},\quad f\in\H(\D);
        $$
        \item[(ii)] $\Bb_0\subset \B^{\b}_0$.
    \end{itemize}
\end{proposition}
\begin{proof}
    To prove $(i)$, let $f\in \Bb$ and set $n=n_\b$. By Lemma~\ref{lemma: control f^n(0)} (i), it suffices to show that
    $$
    \sup_{z\in\D}\abs{f^{(n)}(z)}(1-|z|)^{n-\b} \lesssim \nm{f}_{\Bb},
    $$
    in fact, by the maximum modulus principle, it suffices to consider $|z|\geq \frac{1}{2}$.
    By Theorem~\ref{th:perala}, Lemma~\ref{lemmaletter: expresión derivada fraccionaria} and Lemma~\ref{lemma: control f^n(0)} (i), for every $|z|\geq \frac{1}{2}$,
    \begin{equation}\label{eq:lemma: fn control norma 1}
    \abs{f^{(n)}(z)} \lesssim \nm{f}_{\Bb}  + \int_\D \abs{D^\mu(f)(\z)}\abs{B^{W^n}_z(\z)} V_{\mu_+,n}(\z)dA(\z) ,
    \end{equation}
    so it remains to prove that 
    $$
     \int_\D \abs{D^\mu(f)(\z)}\abs{B^{W^n}_z(\z)} V_{\mu_+,n}(\z)dA(\z) \lesssim \frac{\nm{f}_{\Bb}}{(1-\abs{z})^{n-\b}}.
    $$
    First, by Lemma~\ref{lema pesos tec} (iii)
    \begin{equation}\label{eq:lemma: fn control norma 2}
        \begin{split}
     \int_\D \abs{D^\mu(f)(\z)}\abs{B^{W^n}_z(\z)} V_{\mu_+,n}(\z)dA(\z)
    &\lesssim \int_{\D} \abs{D^\mu(f)(\z)}\abs{B^{W^n}_z(\z)} \mug(\z) (1-|\z|)^n dA(\z) \\ &\leq \nm{f}_{\Bb} \int_{\D} \abs{B^{W^n}_z(\z)}  (1-|\z|)^{n+\b} dA(\z),
        \end{split}
    \end{equation}
    Lastly, by \cite[Theorem 1]{PelRatLiouville} and Lemma~\ref{lema pesos tec}
    \begin{equation}\label{eq:lemma: fn control norma 3}
        \begin{split}
    \int_{\D} \abs{B^{W^n}_z(\z)}  (1-|\z|)^{n+\b} dA(\z) &\lesssim \int_0^{|z|}\frac{(1-t)^{n+\b}}{\widehat{W^n}(t)}dt \asymp  \int_0^{|z|}\frac{dt}{(1-t)^{n-\b+1}} \lesssim \frac{1}{(1-|z|)^{n-\b}}
        \end{split}
    \end{equation}
    Combining \eqref{eq:lemma: fn control norma 1},\eqref{eq:lemma: fn control norma 2} and \eqref{eq:lemma: fn control norma 3} proves $(i)$.

    The proof of $(ii)$ follows a similar argument. Let $f\in\Bb_0\subset \Bb$, by \eqref{eq:lemma: fn control norma 1} and the fact that $n-\b>0$, it suffices to prove that
    $$
    \lim_{|z|\to 1^-} (1-\abs{z})^{n-\b}\int_\D \abs{D^\mu(f)(\z)}\abs{B^{W^n}_z(\z)} V_{\mu_+,n}(\z)dA(\z)  = 0,
    $$
    Given $\e>0$, choose $0<r_0<1$ such that
    $$
    \abs{D^\mu(f)(z)}\mug(z)(1-|z|)^{-\b}<\e,\quad r_0\leq |z|<1.
    $$
    For every $|z|>\frac{1}{2}$ by arguing as in \eqref{eq:lemma: fn control norma 2} and using \eqref{eq:lemma: fn control norma 3}, we obtain
    \begin{equation*}
      \begin{split}
    \int_\D \abs{D^\mu(f)(\z)}\abs{B^{W^n}_z(\z)} V_{\mu_+,n}(\z)dA(\z)  
    &\lesssim \int_{\D} \abs{D^\mu(f)(\z)}\abs{B^{W^n}_z(\z)} \mug(\z) (1-|\z|)^n dA(\z) \\
    &\leq\nm{f}_{\Bb}\int_{0\leq |\z|\leq r_0}\abs{B^{W^n}_z(\z)}(1-|\z|)^{n+\b}dA(\z) \\
    &+ \e\int_{|\z|\geq r_0} \abs{B^{W^n}_z(\z)} (1-|\z|)^{n+\b} dA(\z) \\
    &\leq C(r_0,\mu,n,\b)\nm{f}_{\Bb} + \e\int_{\D} \abs{B^{W^n}_z(\z)} (1-|\z|)^{n+\b} dA(\z) \\
    &\lesssim C(r_0,\mu,n,\b)\nm{f}_{\Bb} + \frac{\e}{(1-|z|)^{n-\b}}.
        \end{split}  
    \end{equation*}
    Since $n-\b>0$, there exists $\frac{1}{2}<r_1<1$ such that
    $$
    (1-|z|)^{n-\b}< \frac{\e}{C(r_0,\mu,n,\b)\nm{f}_{\Bb}},\quad |z|>r_1.
    $$
    It follows that
    $$
    (1-\abs{z})^{n-\b}\int_\D \abs{D^\mu(f)(\z)}\abs{B^{W^n}_z(\z)} V_{\mu_+,n}(\z)dA(\z) \lesssim \e,\quad |z|>r_1,
    $$
    obtaining $(ii)$.
\end{proof}

For the reverse embedding, additional assumptions on $\mu$ are necessary.
\begin{proposition}\label{propo: Bmu<}
    Let $\b>0$ and $\mu$ be a radial weight. Then the following are equivalent:
    \begin{itemize}
        \item[(i)] $\mu\in\DDD$ and $\beta < \b_\mu$;
        \item[(ii)] $\B^{\b}\subset \Bb$, and
        $$
        \nm{f}_{\Bb} \lesssim \nm{f}_{\B^{\b}},\quad f\in\H(\D);
        $$
        \item[(iii)] $\B^{\b}_0\subset \Bb_0$, and
        $$
        \nm{f}_{\Bb} \lesssim \nm{f}_{\B^{\b}},\quad f\in\B^{\b}_0.
        $$
    \end{itemize}
\end{proposition}
\begin{proof}
   We first prove that (i) implies (ii). Let $0<\b<\b_\mu$ and $\mu \in \DDD$. There exists $\lambda>\beta$ for which Lemma~\ref{lemma: caract dcheck} (ii) holds with exponent $\eta=\lambda$, then
    \begin{equation}\label{eq: peso va a cero}
    \lim_{r\to 1^-} \mug(r)(1-r)^{-\b} \leq C \lim_{r\to 1^-} (1-r)^{\lambda-\b} =0,
    \end{equation}
    In particular $\{0\}\subsetneq \Bb_0\subset \Bb$. Set $n=n_\b$, by Lemma~\ref{lemmaletter: expresión derivada fraccionaria}, for every $f\in A^1$,
    \begin{equation}\label{eq: Dmu in terms fn}
        \abs{D^\mu(f)(z)} = \abs{\sum_{j=0}^{n-1} \frac{\widehat{f}(j)}{\mu_{2j+1}}z^j + 4^n z^n R^{W^n,V_{\mu_+,n}}(f^{(n)})(z)}.
    \end{equation}
    Let $f\in \B^{\b}$, by \eqref{eq: L-P estimate}
    $$
    \int_\D \abs{f(z)}dA(z) \asymp \sum_{j=0}^{n-1} \abs{f^{(j)}(0)} + \int_\D \abs{f^{(n)}(z)}(1-\abs{z})^n dA(z) \leq \nm{f}_{\B^\b}\left (1+\int_\D (1-|z|)^\b dA(z) \right ) <\infty.
    $$
    Thus $f\in A^1$. By Lemma~\ref{lemma: control f^n(0)}, for $z\in\D$,
    $$
    \abs{D^\mu(f)(z)} \lesssim \nm{f}_{\B^{\b}} + \abs{R^{W^n,V_{\mu_+,n}}(f^{(n)})(z)}.
    $$
    Hence, for every $z\in\D$,
    \begin{equation}\label{eq: mod argument final}
    \abs{D^\mu(f)(z)}\mug(z)(1-|z|)^{-\b} \lesssim\nm{f}_{\B^{\b}}\mug(z)(1-|z|)^{-\b} + \abs{R^{W^n,V_{\mu_+,n}}(f^{(n)})(z)}\mug(z)(1-|z|)^{-\b},
    \end{equation}
    so it remains to check that 
    $$
    \abs{R^{W^n,V_{\mu_+,n}}(f^{(n)})(z)} \lesssim \frac{(1-|z|)^\b}{\mug(z)} \nm{f}_{\B^{\b}}.
    $$
    First, by Theorem~\ref{th:perala} and Lemma~\ref{lema pesos tec}
    \begin{equation}\label{eq: mod argument final2}
        \begin{split}
        \abs{R^{W^n,V_{\mu_+,n}}(f^{(n)})(z)} &\leq \int_\D \abs{f^{(n)}(\z)} \abs{B^{V_{\mu_+,n}}_z(\z)} W^n(\z)dA(\z) \\
        &\lesssim \nm{f}_{\B^{\b}}\int_\D \abs{B^{V_{\mu_+,n}}_z(\z)} (1-|\z|)^{n+\b}dA(\z).
         \end{split}
    \end{equation}
    Since $\mu\in\DD$, \cite[Theorem 1]{PelRatLiouville}, Lemma~\ref{lema pesos tec} and Lemma~\ref{lemma: caract dcheck} (ii) yield, for $|z|>\frac{1}{2}$,
    \begin{equation}\label{eq: mod argument final3}
        \begin{split}
    \int_\D \abs{B^{V_{\mu_+,n}}_z(\z)} (1-|\z|)^{n+\b}dA(\z) &\lesssim \int_0^{|z|} \frac{(1-t)^{n+\b}}{\widehat{V_{\mu_+,n}}(t)}dt \lesssim \int_0^{|z|} \frac{dt}{\mug(t)(1-t)^{1-\b}} \\
    &\lesssim \frac{(1-|z|)^{\lambda}}{\mug(z)}\int_0^{|z|} \frac{dt}{(1-t)^{1+ (\lambda-\b)}} \lesssim \frac{(1-|z|)^\b}{\mug(z)}, 
        \end{split}
    \end{equation}
    obtaining (ii). To prove that $(i)$ implies $(iii)$, let $f\in\B^{\b}_0$. By \eqref{eq: mod argument final} and \eqref{eq: peso va a cero}, it suffices to show that
    $$
    \lim_{|z|\to 1^-} \abs{R^{W^n,V_{\mu_+,n}}(f^{(n)})(z)}\mug(z)(1-|z|)^{-\b}  = 0.
    $$
    Given $\e>0$, choose $0<r_0<1$ such that
    $$
    \abs{f^{(n)}(z)}(1-|z|)^{n-\b}<\e,\quad r_0\leq |z|<1.
    $$
    For $|z|>\frac{1}{2}$, arguing as in \eqref{eq: mod argument final2} and using \eqref{eq: mod argument final3}, we obtain
    \begin{equation*}
      \begin{split}
    \abs{R^{W^n,V_{\mu_+,n}}(f^{(n)})(z)} &\lesssim \nm{f}_{\B^{\b}}\int_{0\leq |\z|\leq r_0} \abs{B^{V_{\mu_+,n}}_z(\z)} (1-|\z|)^{n+\b}dA(\z) \\
    &+ \e \int_{|\z|\geq r_0} \abs{B^{V_{\mu_+,n}}_z(\z)} (1-|\z|)^{n+\b}dA(\z) \\
    &\leq C(r_0,\mu,n,\beta) \nm{f}_{\B^{\b}} + \e \int_\D \abs{B^{V_{\mu_+,n}}_z(\z)} (1-|\z|)^{n+\b}dA(\z)\\
    &\lesssim C(r_0,\mu,n,\beta) \nm{f}_{\B^{\b}} + \e  \frac{(1-|z|)^\b}{\mug(z)}.
        \end{split}  
    \end{equation*}
    On the other hand, by \eqref{eq: peso va a cero}, there exists $\frac{1}{2}<r_1<1$ such that
    $$
    \mug(z)(1-|z|)^{-\b}< \frac{\e}{C(r_0,\mu,n,\beta)\nm{f}_{\B^{\beta}}},\quad |z|>r_1
    $$
    then we conclude that
    $$
    \abs{R^{W^n,V_{\mu_+,n}}(f^{(n)})(z)}\mug(z)(1-|z|)^{-\b} \lesssim \e,\quad |z|>r_1,
    $$
    which proves $(iii)$.

To finish the proof, we show that $(ii)$ implies $(i)$. We first check that $\mu\in\DD$. Testing $(ii)$ on the monomials $z^m$, $m\in\N$, and writing $n=n_\b$, then
    $$
    \nm{z^m}_{{\B^\b}}= \begin{cases}
 m!& \text{if}\;m\leq n \\
\frac{m!}{(m-n)!}\left (\frac{m-n}{m-\b} \right )^{m-n}\left (\frac{n-\b}{m-\b} \right )^{n-\b} & \text{if}\;m>n
\end{cases}
    $$
    Now we notice that, as $m\to \infty$,
    $$
    \frac{m!}{(m-n)!}\asymp m^n,\quad \left (\frac{n-\b}{m-\b} \right )^{n-\b}\asymp m^{\b-n} \quad \text{and} \quad \lim_{m\to \infty} \left (\frac{m-n}{m-\b} \right )^{m-n}= e^{\b-n},
    $$
    and hence, as $m\to \infty$
    $$
    \nm{z^m}_{{\B^\b}} \asymp m^\b.
    $$
    Then, by (ii), for every $0\leq r <1$
    \begin{equation}\label{eq: necesidad 2}
    \mug(r)r^m\lesssim m^\b\mu_{2m+1}(1-r)^\b.
    \end{equation}
    Therefore, integrating by parts
    \begin{equation*}
    \begin{split}
    \mu_{\frac{3}{2}m+1} &= \left (\frac{3}{2}m+1 \right )\int_0^1 r^{\frac{3}{2}m}\mug(r)dr \lesssim\left (\frac{3}{2}m+1 \right )\int_0^1 r^{\frac{m}{2}}m^\b\mu_{2m+1}(1-r)^\b dr \\ &\asymp \mu_{2m+1}m^{\b+1} \int_0^1 r^{\frac{m}{2}}(1-r)^\b dr = \mu_{2m+1}m^{\b+1}  B\left (\frac{m}{2}+1,\b+1 \right )\\
    &= \mu_{2m+1}m^{\b+1}  \frac{\Gamma(\frac{m}{2}+1)\Gamma(\b+1)}{\Gamma(\frac{m}{2}+\b+2)},
    \end{split}
    \end{equation*}
    and by standard estimates when $m\to \infty$
    $$
    \frac{\Gamma(\frac{m}{2}+1)}{\Gamma(\frac{m}{2}+\b+2)} \asymp \left (\frac{m}{2}+1 \right )^{-(\b+1)} \asymp m^{-(\b+1)},
    $$
    combining this with the preceding estimate gives
    $$
    \mu_{\frac{3}{2}m+1} \lesssim \mu_{2m+1},
    $$
    and standard arguments show that $\mu\in\DD$.

    Next, we prove that there exist $\e>0$ and $C>0$ such that
    \begin{equation}\label{eq: momentos M}
    \mu_y\leq C \left ( \frac{x}{y} \right )^{\b+\e} \mu_x,\quad 1\leq x \leq y<\infty.
    \end{equation}

Consider the polynomials
$$
    F_N(z)=\sum_{k=1}^{N}\frac{z^k}{k^{\b+1}},
    \qquad N\geq2.
$$
Each $F_N$ belongs to $\B^\b_0$, and
\begin{equation}\label{eq: necessity test polynomial norms}
    \sup_{N\geq2}\nm{F_N}_{\B^\b}<\infty.
\end{equation}
Indeed, writing $n=n_\b>\b$, the derivatives
$F_N^{(j)}(0)$, $0\leq j<n$, are uniformly bounded.
On the other hand,
$$
    \abs{F_N^{(n)}(z)}
    \leq
    \sum_{k=n}^{N}
    \frac{k!}{(k-n)!}\frac{|z|^{k-n}}{k^{\b+1}} \leq \sum_{k=n}^{N}
    |z|^{k-n}k^{n-\b-1},
$$
For $|z|\leq\frac12$, we have $\abs{F_N^{(n)}(z)} \leq \sum_{k=0}^{\infty}2^{-k}<\infty$, while for $\frac{1}{2}<|z|<1$, 
$$
\abs{F_N^{(n)}(z)} \leq 2^n \sum_{k=n}^{\infty} |z|^k k^{n-\b-1} \lesssim \sum_{k=0}^{\infty} |z|^{k}\frac{\Gamma(k+n-\b)}{\Gamma(k+1)} \leq \Gamma(n-\b)(1-|z|)^{\b-n}
$$
where the last step uses the generalized binomial theorem: for $a>0$ and $0\leq r <1$
$$
\frac{1}{(1-r)^{a}} = \sum_{k=0}^\infty \frac{\Gamma(k+a)}{\Gamma(a)\Gamma(k+1)}r^k.
$$
Thus \eqref{eq: necessity test polynomial norms}. Applying (ii) to these polynomials gives, for every $0\leq r <1$,
$$
    (1-r)^{-\b}\mug(r)
    \sum_{k=1}^{N}
    \frac{r^k}{k^{\b+1}\mu_{2k+1}}
    \lesssim1.
$$
In particular, set $r_n=1-\frac{1}{n}$ for $n\geq 2$. Evaluating at $r_N$, applying Lemma~\ref{lemma: caract Dgorro} (iv), and using the fact that $r_N^k$ is uniformly bounded below for $1\leq k\leq N$, we obtain 
\begin{equation}\label{eq: necessity moment sum}
    \sum_{k=1}^{N}\frac{1}{k^{\b+1}\mu_k}
    \lesssim\frac{1}{N^\b\mu_N}.
\end{equation}
Evaluating \eqref{eq: necesidad 2} at $r_k$ and applying Lemma~\ref{lemma: caract Dgorro} (iii) and (iv) gives
\begin{equation}\label{eq: necessity inverse moments}
  \mu_k k^\b \lesssim \mu_m m^\b \left (1-\frac{1}{k} \right )^{-m} \lesssim  \mu_m m^\b,\quad 2\leq m < k <\infty. 
\end{equation}
Combining \eqref{eq: necessity moment sum} and \eqref{eq: necessity inverse moments}, we obtain, for $N>m\geq2$,
\begin{align*}
    \frac{1}{N^\b\mu_N}
    &\gtrsim
    \sum_{k=m}^{N}\frac{1}{k^{\b+1}\mu_k}\\
    &\gtrsim
    \frac{1}{m^\b\mu_m}
    \sum_{k=m}^{N}\frac1k\\
    &\geq  \frac{1}{m^\b\mu_m}\sum_{k=m}^{N}\int_k^{k+1} \frac{1}{x}dx \\
    &\geq \frac{\log(N/m)}{m^\b\mu_m}
\end{align*}
Consequently, there exists a constant $C_1>0$
such that
    $$
    \mu_N\leq
    \frac{C_1}{\log(N/m)}
    \left(\frac mN\right)^\b\mu_m,
    \qquad N>m\geq2.
    $$
Now choose $K\in\N\setminus\{1\}$, such that $C_1/\log K\leq\frac12$ and set
$\e=\log2/\log K>0$. Since $K^{-\e} = \frac{1}{2}$, taking $N=Km$ gives
$$
    \mu_{Km}\leq\frac12K^{-\b}\mu_m
    =K^{-(\b+\e)}\mu_m,
    \qquad m\geq2.
$$
Iterating this inequality gives
    $$\mu_{K^jm}\leq
    K^{-j(\b+\e)}\mu_m,
    \qquad j\geq0,\quad m\geq2.$$
Let $n\geq m\geq2$ be integers, and choose $j\in \N\cup\{0\}$ such that
$K^jm\leq n<K^{j+1}m$. Then $K^{-j}< K\frac{m}{n}$, and by the monotonicity of the moments
$$
    \mu_n\leq\mu_{K^jm}
    \leq K^{-j(\b+\e)}\mu_m
    \leq K^{\b+\e}
    \left(\frac mn\right)^{\b+\e}\mu_m,
$$
This proves \eqref{eq: momentos M} for any pair of integers $n\geq m\geq 2$. Monotonicity and standard arguments extend the estimate to all real $1\leq x\leq y <\infty$. Finally, since $\mu\in\DD$, Lemma~\ref{lemma: caract Dgorro} (iv) and \eqref{eq: momentos M} imply
$$
    \mug(s)\lesssim
    \left(\frac{1-s}{1-t}\right)^{\b+\e}\mug(t),\quad 0\leq t\leq s <1,
$$
so $\mu\in\Dd$ and $\b<\b_\mu$.

All the test functions used above belong to $\B^{\beta}_0$, so the same argument proves that $(iii)$ implies $(i)$, concluding the proof.
\end{proof}
\section{Essential norm}\label{Sec5}
In this section, we establish the distance formulas needed to estimate the essential norm of $V_{\mu,g}$ and the proofs of Theorems~\ref{th: norma esencial p=q} and \ref{th: caract Vmu p<q}.
\begin{lemma}\label{lemma: distancia Bb}
    Let $\b\geq 0$ and $\mu$ be a radial weight. If $f\in \Bb$, then  
        $$
        \dist(f,\Bb_0) \asymp \limsup_{|z|\to 1^-}\abs{D^\mu(f)(z)}\mug(z)(1-|z|)^{-\b}.
        $$
\end{lemma}
\begin{proof}
    We distinguish three cases according to the value of 
    $$
    L_{\b,\mu} = \limsup_{r\to 1^-} \frac{\mug(r)}{(1-r)^\b}.
    $$
    If $L_{\b,\mu}=\infty$, then $\Bb=\Bb_0=\{0\}$ and the result is trivial, while if $0<L_{\b,\mu}<\infty$, then $\Bb=\{f\in\H(\D): D^\mu(f)\in H^\infty \}$ and $\Bb_0=\{0\}$, so
    $$
    \dist(f,\Bb_0) = \nm{f}_{\Bb}.
    $$
    We claim that $\nm{f}_{\Bb} \asymp \nm{D^{\mu}(f)}_{H^\infty}$. On the one hand,  
    $$
    \nm{f}_{\Bb} \leq \nm{D^{\mu}(f)}_{H^\infty}\sup_{0\leq r <1}\mug(r)(1-r)^{-\b},
    $$
    on the other hand, for fixed $0\leq s<1$ and every $s<r<1$,
    $$
    M_\infty(s,D^\mu(f)) \mug(r)(1-r)^{-\b} \leq M_\infty(r,D^\mu(f)) \mug(r)(1-r)^{-\b} \leq \nm{f}_{\Bb},
    $$
    Taking the limit superior as $r\to 1^-$ and then letting $s\to 1^-$ yields
    $$
    L_{\b,\mu}\nm{D^\mu(f)}_{H^\infty} \leq \nm{f}_{\Bb}.
    $$
    Combining these inequalities, we obtain
    $$
    \dist(f,\Bb_0)\asymp \nm{D^\mu(f)}_{H^\infty} = \limsup_{|z|\to 1^-} \abs{D^\mu(f)(z)}=\frac{1}{L_{\b,\mu}} \limsup_{r\to 1^-} M_\infty(r,D^\mu(f))\mug(r)(1-r)^{-\b}.
    $$
    Finally, suppose that $L_{\b,\mu}=0$, and let $f\in\Bb$. One inequality is immediate: for $h\in\Bb_0$ and $z\in\D$, the linearity of $D^\mu$ gives
    $$
    \abs{D^\mu(f)(z)}\mug(z)(1-|z|)^{-\b} \leq \nm{f-h}_{\Bb} + \abs{D^\mu(h)(z)}\mug(z)(1-|z|)^{-\b},
    $$
    and taking the limit superior gives
    $$
    \limsup_{|z|\to 1^-}\abs{D^\mu(f)(z)}\mug(z)(1-|z|)^{-\b} \leq \dist(f,\Bb_0).
    $$
    
    For the reverse inequality, note that $L_{\b,\mu}=0$ implies that $f_r(z)=f(rz)$ belongs to $\Bb_0$ for every $0\leq r <1$. Hence
    $$
    \dist(f,\Bb_0)\leq \nm{f-f_r}_{\Bb} = \sup_{z\in\D} \abs{D^\mu(f-f_r)(z)}\mug(z)(1-|z|)^{-\b}.
    $$
    Fix $0\leq R <1$. Since $D^\mu(f-f_r) = D^\mu(f)-D^\mu(f)_r$, we obtain that $D^\mu(f-f_r)$ converges to zero uniformly on compact subsets as $r\to 1^-$. Thus, 
    \begin{equation}\label{eq: expdist1}
    \lim_{r\to 1^-} \sup_{|z|\leq R} \abs{D^\mu(f-f_r)(z)}\mug(z)(1-|z|)^{-\b} = 0.
    \end{equation}
    For the exterior region first by the monotonicity of the integral means gives, for $R< s < 1$,
    $$
    M_\infty(s,D^\mu(f-f_r)) \leq M_\infty(s,D^\mu(f)) + M_\infty(rs,D^\mu(f)) \leq 2M_\infty(s,D^\mu(f)),
    $$
    therefore, for every $0\leq r <1$,
    $$
    \sup_{|z|> R} \abs{D^\mu(f-f_r)(z)}\mug(z)(1-|z|)^{-\b} \leq 2 \sup_{|z|> R} \abs{D^\mu(f)(z)}\mug(z)(1-|z|)^{-\b}.
    $$
    Combining this estimate with \eqref{eq: expdist1} and taking the limit superior as $r\to 1^-$ gives
    $$
    \dist(f,\Bb_0)\leq \limsup_{r\to 1^-} \nm{f-f_r}_{\Bb} \leq 2\sup_{|z|> R} \abs{D^\mu(f)(z)}\mug(z)(1-|z|)^{-\b},
    $$
    and the proof concludes by letting $R\to 1^-$.  
\end{proof}

To study the essential norm of $V_{\mu,g}$ on $H^p$, we introduce, for a radial weight $\mu$, the space $\BMOA_\mu$ of $g\in \H(\D)$ such that
 $$\|g\|^2_{\BMOA_\mu}=\sup_{a\in\D}\frac{\int_{S(a)}\abs{D^\mu(g)(z)}^2\frac{\mug(z)^2}{1-\abs{z}}dA(z)}{1-\abs{a}}<\infty$$ and $g\in\VMOA_\mu$ if  
 $$
 \lim_{|a|\to 1^-}\frac{\int_{S(a)}\abs{D^\mu(g)(z)}^2\frac{\mug(z)^2}{1-\abs{z}}dA(z)}{1-\abs{a}}=0
 $$ 
\begin{lemma}\label{lemma: distancia BMOAmu}
    Let $\mu$ be a radial weight. If $f\in \BMOA_\mu$, then
        $$\dist(f,\VMOA_\mu)\asymp\limsup_{|a|\to 1^-}\left ( \frac{\int_{S(a)}\abs{D^\mu(f)(z)}^2\frac{\mug(z)^2}{1-|z|}dA(z)}{1-|a|}\right )^{\frac{1}{2}}.
        $$
\end{lemma}
\begin{proof}
    The first inequality follows as in the previous lemma. If $f\in \BMOA_\mu$, then, for every $h\in \VMOA_\mu$,
    $$
    \left ( \frac{\int_{S(a)}\abs{D^\mu(f)(z)}^2\frac{\mug(z)^2}{1-|z|}dA(z)}{1-|a|}\right )^{\frac{1}{2}} \leq \nm{f-h}_{\BMOA_\mu} + \left( \frac{\int_{S(a)}\abs{D^\mu(h)(z)}^2\frac{\mug(z)^2}{1-|z|}dA(z)}{1-|a|}\right )^{\frac{1}{2}},
    $$
    taking the limit superior gives
    $$
    \limsup_{|a|\to 1^-}\left ( \frac{\int_{S(a)}\abs{D^\mu(f)(z)}^2\frac{\mug(z)^2}{1-|z|}dA(z)}{1-|a|}\right )^{\frac{1}{2}} \leq \dist(f,\VMOA_\mu).
    $$
    For the reverse inequality, we may assume that $\int_0^1 \frac{\mug(r)^2}{1-r}dr<\infty$: otherwise, $\BMOA_\mu=\VMOA_\mu=\{0\}$ and the result is trivial (see the proof of \cite[Proposition 3.1]{BelMorNikPel}). Thus 
    $$
    \lim_{|a|\to 1^-}\int_{|a|}^1 \frac{\mug(r)^2}{1-r}dr=0,
    $$
    so $f_r\in \VMOA_\mu$ for every $f\in\H(\D)$ and $0\leq r <1$. Consequently,
    $$
    \dist(f,\VMOA_\mu) \leq \nm{f-f_r}_{\BMOA_\mu} = \sup_{a\in\D}\left ( \frac{\int_{S(a)}\abs{D^\mu(f-f_r)(z)}^2\frac{\mug(z)^2}{1-|z|}dA(z)}{1-|a|}\right )^{\frac{1}{2}}.
    $$
    Fix $0\leq R <1$. For every $|a|\leq R$,
    $$
    \frac{\int_{S(a)}\abs{D^\mu(f-f_r)(z)}^2\frac{\mug(z)^2}{1-|z|}dA(z)}{1-|a|} \leq \frac{1}{1-R} \int_\D \abs{D^\mu(f-f_r)(z)}^2\frac{\mug(z)^2}{1-|z|}dA(z)
    $$
    Since $f\in \BMOA_\mu$, then $D^\mu(f)\in A^2_\om$, where $\om(z)=\frac{\mug(z)^2}{1-|z|}$. As $\om$ is radial, dilations converge in norm, and hence $\lim_{r\to 1^-}\nm{D^\mu(f-f_r)}_{A^2_\om}^2 = \lim_{r\to 1^-}\nm{D^\mu(f)-D^\mu(f)_r}^2_{A^2_\om}=0$. Therefore, for every $0\leq R<1$,
    $$
    \lim_{r\to 1^-} \sup_{|a|\leq R}\left ( \frac{\int_{S(a)}\abs{D^\mu(f-f_r)(z)}^2\frac{\mug(z)^2}{1-|z|}dA(z)}{1-|a|}\right )^{\frac{1}{2}} = 0.
    $$
    Consequently, for every $0<R<1$,
    \begin{equation}\label{eq:lemmadist2}
    \limsup_{r\to 1^-} \nm{f-f_r}_{\BMOA_\mu} \leq  \limsup_{r\to 1^-} \sup_{|a|>R}\left ( \frac{\int_{S(a)}\abs{D^\mu(f-f_r)(z)}^2\frac{\mug(z)^2}{1-|z|}dA(z)}{1-|a|}\right )^{\frac{1}{2}},
    \end{equation}
    to complete the proof, we show that 
    \begin{equation}\label{eq:lemmadist3}
    \sup_{|a|>R}\left ( \frac{\int_{S(a)}\abs{D^\mu(f-f_r)(z)}^2\frac{\mug(z)^2}{1-|z|}dA(z)}{1-|a|}\right )^{\frac{1}{2}}\leq 2\sup_{|a|>R}\left ( \frac{\int_{S(a)}\abs{D^\mu(f)(z)}^2\frac{\mug(z)^2}{1-|z|}dA(z)}{1-|a|}\right )^{\frac{1}{2}}.
    \end{equation}
    By the triangle inequality, it suffices to show that 
    \begin{equation}\label{eq:lemmadist4}
    \sup_{|a|>R}\left ( \frac{\int_{S(a)}\abs{D^\mu(f_r)(z)}^2\frac{\mug(z)^2}{1-|z|}dA(z)}{1-|a|}\right )^{\frac{1}{2}}\leq \sup_{|a|>R}\left ( \frac{\int_{S(a)}\abs{D^\mu(f)(z)}^2\frac{\mug(z)^2}{1-|z|}dA(z)}{1-|a|}\right )^{\frac{1}{2}}.
    \end{equation}
    By the Poisson integral representation,
    $$
    \abs{D^\mu(f_r)(z)}^2 \leq \frac{1}{2\pi} \int_0^{2\pi} \frac{1-r^2}{|1-re^{-i\t}|^2}\abs{D^\mu(f)(ze^{i\t})}^2d\t,
    $$
    Fubini's theorem and a change of variables give, for every $|a|>R$,
    \begin{equation*}
    \begin{split}
    \int_{S(a)}\abs{D^\mu(f_r)(z)}^2\frac{\mug(z)^2}{1-|z|}dA(z) &\leq \int_{S(a)}\left (\frac{1}{2\pi} \int_0^{2\pi} \frac{1-r^2}{|1-re^{-i\t}|^2}\abs{D^\mu(f)(ze^{i\t})}^2d\t \right )\frac{\mug(z)^2}{1-|z|}dA(z) \\
    &= \frac{1}{2\pi} \int_0^{2\pi} \frac{1-r^2}{|1-re^{-i\t}|^2} \left ( \int_{S(e^{i\t}a)}\abs{D^\mu(f)(z)}^2\frac{\mug(z)^2}{1-|z|}dA(z) \right ) d\t.
    \end{split}
    \end{equation*}
    Dividing by $1-|a|=1-|e^{i\t}a|$, taking the supremum over $|a|>R$, and using 
    $$
    \frac{1}{2\pi}\int_0^{2\pi} \frac{1-r^2}{|1-re^{-i\t}|^2}d\t=1,
    $$
    proves \eqref{eq:lemmadist4}. Combining this with \eqref{eq:lemmadist2} and \eqref{eq:lemmadist3}, we obtain, for every $0<R<1$
    $$
    \dist(f,\VMOA_\mu) \leq 2\sup_{|a|>R}\left ( \frac{\int_{S(a)}\abs{D^\mu(f)(z)}^2\frac{\mug(z)^2}{1-|z|}dA(z)}{1-|a|}\right )^{\frac{1}{2}},
    $$
    letting $R\to 1^-$ completes the proof.
\end{proof}
Analogous distance formulas hold for $\B^\b$ and $\BMOA$. The following statements can be proved using the same techniques as in the preceding lemmas; see also \cite{TjaniBloch}.
\begin{letterlemma}\label{lemma: distancia Bbn}
Let $\b\geq 0$ and $f\in\H(\D)$. Then
\begin{itemize}
    \item[(i)] If $f\in \B^\b$, then  
        $$
        \dist(f,\B^\b_0) \asymp \limsup_{|z|\to 1^-}\abs{f^{(n_\b)}(z)}(1-|z|)^{n_\b-\b};
        $$
    \item[(ii)] If $f\in \BMOA$, then
        $$
        \dist(f,\VMOA) \asymp \limsup_{|a|\to 1^-}\left ( \frac{\int_{S(a)}\abs{f'(z)}^2 (1-|z|)dA(z)}{1-|a|}\right )^{\frac{1}{2}}
        $$
\end{itemize}
\end{letterlemma}
Now we begin the study of the essential norm. The following technical lemma is central to the proofs of the main results.
\begin{lemma}\label{lemma: compact trick}
    Let $\a\geq -1$, $0<p\leq q<\infty$, $\mu\in\DDD$ and $g\in\H(\D)$, such that $V_{\mu,g}:A^p_\a \to A^q_\a$ is bounded. If a family $\{f_a\}_{a\in\D}\subset A^p_\a$ satisfies
    $$
    \sup_{a\in\D}\nm{f_a}_{A^p_\a}\leq 1, \quad f_a\to 0 \text{ uniformly on compact subsets as $|a|\to 1^-$},
    $$
    then there exists a constant $C=C(q)>0$ such that
    $$
    \limsup_{|a|\to 1^-} \nm{V_{\mu,g}(f_{a})}_{A^q_\a} \leq C \nm{V_{\mu,g}}_{e,A^p_\a\to A^q_\a}.
    $$
\end{lemma}
\begin{proof}
    Let $\{f_a\}_{a\in\D}$ satisfying the hypothesis. We first show that $V_{\mu,g}(f_a)\to 0$ uniformly on compact subsets as $|a|\to 1^-$. For $f\in\H(\D)$ and $z\in\D$, we have
    $$
    I^\mu(f)(z) = \sum_{n=0}^\infty \mu_{2n+1}\widehat{f}(n)z^n =\sum_{n=0}^\infty \int_0^1\widehat{f}(n)(zr^2)^n r\mu(r) dr = \int_0^1 f(zr^2)r\mu(r)dr.
    $$
    Fix $0\leq R <1$. Applying this identity to $V_{\mu,g}(f_a)=I^\mu(f_a D^\mu(g))$, gives, for $z\in \overline{D(0,R)}$,
    $$
    \abs{V_{\mu,g}(f_a)(z)}\leq\int_0^1 \abs{f_a(zr^2)} \abs{D^\mu(g)(zr^2)}r\mu(r)dr \leq \mu_1\sup_{|\z|\leq R} \abs{f_a(\z)}\sup_{|\z|\leq R} \abs{D^\mu(g)(\z)}, 
    $$
    then the convergence of $f_a$ to zero uniformly on compact subsets implies the same convergence for $V_{\mu,g}(f_a)$.

    Now, let $K$ be a compact operator and $\{a_n\}_{n\in\N}$ a sequence such that $|a_n|\to 1^-$ and
    $$
    \lim_{n\to \infty} \nm{V_{\mu,g}(f_{a_n})}_{A^q_\a} =\limsup_{|a|\to 1^-} \nm{V_{\mu,g}(f_a)}_{A^q_\a}.
    $$
    By compactness there exists a subsequence, still denoted by $\{f_{a_n}\}_{n\in\N}$, such that $\{K(f_{a_n})\}_{n\in\N}\to h$ in $A^q_\a$ for some $h\in A^q_\a$. Thus $K(f_{a_n})-V_{\mu,g}(f_{a_n})\to h$ uniformly on compact subsets. Fatou's lemma yields
    \begin{equation*}
    \nm{h}_{A^q_\a} \leq \liminf_{n\to \infty}\nm{(V_{\mu,g}-K)(f_{a_n})}_{A^q_\a} 
    \leq \nm{V_{\mu,g}-K}_{A^p_\a\to A^q_\a}.
    \end{equation*}
    Lastly, for every $n\in \N$
    $$
    \nm{V_{\mu,g}(f_{a_n})}_{A^q_\a} \lesssim \nm{(V_{\mu,g}-K)(f_{a_n})}_{A^q_\a} + \nm{K(f_{a_n})}_{A^q_\a} \leq\nm{V_{\mu,g}-K}_{A^p_\a\to A^q_\a} + \nm{K(f_{a_n})}_{A^q_\a},
    $$
    taking limits gives
    $$
    \limsup_{|a|\to 1^-} \nm{V_{\mu,g}(f_a)}_{A^q_\a} \lesssim \nm{V_{\mu,g}-K}_{A^p_\a\to A^q_\a} + \nm{h}_{A^q_\a}\leq 2\nm{V_{\mu,g}-K}_{A^p_\a\to A^q_\a},
    $$
    taking the infimum over all compact operators $K:A^p_\a\to A^q_\a$ completes the proof.
\end{proof}

\begin{proposition}\label{propo: norm essential 1}
    Let $0<p\leq q<\infty$, $\mu\in \DDD$ and $g\in\H(\D)$. Then 
    \begin{itemize}
        \item[(i)] If $p<q$ and $V_{\mu,g}:H^p\to H^q$ is bounded, i.e., $g\in\Bb$ with $\b=\frac{1}{p}-\frac{1}{q}$ then 
        $$
        \nm{V_{\mu,g}}_{e,H^p\to H^q} \asymp \dist(g,\Bb_0)\asymp\limsup_{|z|\to 1^-}\abs{D^\mu(g)(z)}\mug(z)(1-|z|)^{-\b};
        $$
        \item[(ii)] If $p\leq q$, $\a>-1$ and $V_{\mu,g}:A^p_\a\to A^q_\a$ is bounded, i.e., $g\in\Bb$ with $\b=(\a+2)(\frac{1}{p}-\frac{1}{q})$ then 
        $$
        \nm{V_{\mu,g}}_{e,A^p_\a\to A^q_\a} \asymp \dist(g,\Bb_0)\asymp\limsup_{|z|\to 1^-}\abs{D^\mu(g)(z)}\mug(z)(1-|z|)^{-\b}.
        $$
    \end{itemize}
\end{proposition}
\begin{proof}
    We prove both statements together. Write $H^p=A^p_{-1}$ and $\b=(\a+2)(\frac{1}{p}-\frac{1}{q})$ for $\a\geq -1$. If $\a=-1$, we consider only $p<q$, as in $(i)$; for $\a>-1$, we allow $p\leq q$, as in $(ii)$.

    Let $g\in \Bb$. In both cases, the second equivalence follows from Lemma~\ref{lemma: distancia Bb}. It remains to prove
    \begin{equation}\label{eq: equiv distancia 1}
    \limsup_{|z|\to 1^-}\abs{D^\mu(g)(z)}\mug(z)(1-|z|)^{-\b} \lesssim \nm{V_{\mu,g}}_{e,A^p_\a\to A^q_\a} \lesssim \dist(g,\Bb_0).
    \end{equation}

    The second inequality follows from Proposition~\ref{propo: caract Vmu p<q} when $p<q$, and from \cite[Theorem 1.4]{BelMorNikPel} when $p=q$ and $\a>-1$. Indeed, $V_{\mu,h}$ is compact for every $h\in \Bb_0$. The linearity in the symbol and the corresponding norm estimate give
    $$
    \nm{V_{\mu,g}}_{e,A^p_\a \to A^q_\a}\leq \nm{V_{\mu,g}-V_{\mu,h}}_{A^p_\a \to A^q_\a} = \nm{V_{\mu,g-h}}_{A^p_\a \to A^q_\a} \asymp \nm{g-h}_{\Bb},
    $$
    and taking the infimum over $h\in \Bb_0$ proves the inequality. To prove the remaining inequality in \eqref{eq: equiv distancia 1}, consider the test functions,
    $$
    f_{a,\a,p}(z)=\frac{(1-|a|)^{\frac{\a+2}{p}}}{(1-\overline{a}z)^{\frac{2(\a+2)}{p}}},\quad a\in\D.
    $$
    These satisfy $\sup_{a\in\D}\nm{f_{a,\a,p}}_{A^p_\a}\lesssim 1$, and converge to zero uniformly on compact subsets as $|a|\to 1^-$. With $\a$ and $p$ fixed, simply write $f_a=f_{a,\a,p}$. Applying Lemma~\ref{lemma: compact trick} to this family of functions, after a normalization independent of $a$, gives
    $$
    \limsup_{|a|\to 1^-} \nm{V_{\mu,g}(f_{a})}_{A^q_\a} \leq C \nm{V_{\mu,g}}_{e,A^p_\a\to A^q_\a}.
    $$
    It remains to show that
    \begin{equation*}
    \limsup_{|a|\to 1^-}\abs{D^\mu(g)(a)}\mug(a)(1-|a|)^{-\b} \lesssim \limsup_{|a|\to 1^-} \nm{V_{\mu,g}(f_{a})}_{A^q_\a}.
    \end{equation*}
    In fact, we prove the pointwise estimate
    \begin{equation}\label{eq: pointwise}
        \abs{D^\mu(g)(a)}\mug(a)(1-|a|)^{-\b} \lesssim \nm{V_{\mu,g}(f_{a})}_{A^q_\a},
    \end{equation}
    for every $a\in\D$, treating $\a>-1$ and $\a=-1$ separately.
    
    Suppose first that $\a>-1$. First recall the well-known inequality
    \begin{equation}\label{eq: HL ineq}
    M_\infty(r,f)\lesssim \frac{M_q\left (\frac{1+r}{2},f \right )}{(1-r)^{\frac{1}{q}}},\quad f\in\H(\D),
    \end{equation}
    applying \eqref{eq: HL ineq}, \cite[Theorem 1]{PelRosa} and the fact that $\mu\in\DDD$, we obtain
    \begin{equation*}
        \begin{split}
             \abs{D^\mu(g)(a)}^q\mug(a)^q(1-|a|)^{-q\b} &\asymp \abs{D^\mu(g)(a)f_{a}(a)}^q\mug(a)^q(1-|a|)^{\a+2} \\
             &\lesssim M_q^q\left (\frac{1+|a|}{2},D^\mu(g)f_a \right )\mug(a)^q(1-|a|)^{\a+1}  \\
            &\lesssim \int_{\frac{1+|a|}{2}}^1 M_q^q(r,D^\mu(g)f_a)\mug(r)^q(1-r)^{\a} dr\\
            &\lesssim \int_\D \abs{D^\mu(V_{\mu,g}(f_{a}))(z)}^q\mug(z)^q(1-|z|)^{\a} dA(z)\\
            &\asymp \nm{V_{\mu,g}(f_{a})}_{A^q_\a}^q.
        \end{split}
    \end{equation*}
    
    On the other hand, if $\a=-1$ \cite[Theorem 1.1]{BelMorNikPel} gives
    $$
    \nm{V_{\mu,g}(f_{a})}_{H^q} \asymp \nm{D^\mu(g)f_a}_{AT^q_2(\frac{\mug^2}{1-|\cdot|})}
    $$ where $AT^q_2(\om)$ is the weighted analytic tent space induced by  $\om$ studied in \cite{AguiMasPelRat}. Lemma~\ref{lemma: caract dcheck} (ii) shows that $\om(z)=\frac{\mug(z)^2}{1-|z|}$ is a radial weight. Moreover, 
    $$
    \omg(r)=\int_r^1 \frac{\mug(s)^2}{1-s}ds \lesssim \frac{\mug(r)^2}{(1-r)^{2\eta}} \int_r^1 (1-s)^{2\eta-1}ds \asymp \mug(r)^2,
    $$
    and the reverse inequality follows from $\mu\in\DD$:
    $$
    \omg(r)=\int_r^1 \frac{\mug(s)^2}{1-s}ds \geq \int_r^{\frac{1+r}{2}} \frac{\mug(s)^2}{1-s}ds \geq \mug\left (\frac{1+r}{2}\right )^2 \log(2) \gtrsim \mug(r)^2.
    $$
    Thus $\om\in\DDD\subset \DD$, then an application of \cite[Lemma 3.1]{AguiMasPelRat} shows that for every $a\in\D$,
    $$
        \nm{V_{\mu,g}(f_{a})}_{H^q} \asymp \nm{D^\mu(g)f_a}_{AT^q_2(\om)} \gtrsim \abs{D^\mu(g)(a)f_a(a)}\omg(a)^{\frac{1}{2}}(1-|a|)^{\frac{1}{q}} \asymp \abs{D^\mu(g)(a)}\mug(a)(1-|a|)^{-\b}.
    $$
    Concluding the proof.
\end{proof}
Now we can prove Theorem~\ref{th: caract Vmu p<q}.

\begin{Prf}{\em{  Theorem~\ref{th: caract Vmu p<q}}.}
$(i)$ follows from Proposition~\ref{propo: caract Vmu p<q} (i), while $(ii)$ follows from Proposition~\ref{propo: caract Vmu p<q} $(ii)$ and Proposition~\ref{propo: norm essential 1}.

With $(i)$ and $(ii)$ established, $(iii)$ follows from Theorem~\ref{th: caract Bmu}. Finally, $(ii)$, Theorem~\ref{th: caract Bmu} and Lemma~\ref{lemma: distancia Bbn} $(i)$ give
\begin{equation*}
\begin{split}
\nm{V_{\mu,g}}_{e,A^p_\a\to A^q_\a} &\asymp \dist(g,\Bb_0) = \inf_{h\in\Bb_0}\nm{g-h}_{\Bb} \asymp \inf_{h\in\B^\b_0}\nm{g-h}_{\B^\b} \\
&= \dist(g,\B^\b_0) \asymp \limsup_{|z|\to 1^-}\abs{g^{(n_\b)}(z)}(1-|z|)^{n_\b-\b}. 
\end{split}
\end{equation*}
which completes the proof of (iv).
\end{Prf}

Now it remains to deal with the essential norm of $V_{\mu,g}$ acting on $H^p$.

\begin{proposition}\label{propo: norm essential 2}
    Let $0<p<\infty$, $\mu\in \DDD$ and $g\in \BMOA_\mu$. Then
    $$
    \nm{V_{\mu,g}}_{e,H^p\to H^p} \asymp \dist(g,\VMOA_\mu) \asymp \limsup_{|a|\to 1^-}\left ( \frac{\int_{S(a)}\abs{D^\mu(g)(z)}^2\frac{\mug(z)^2}{1-|z|}dA(z)}{1-|a|}\right )^{\frac{1}{2}}.
    $$
\end{proposition}
\begin{proof}
    Let $g\in \BMOA_\mu$, the second equivalence follows from Lemma~\ref{lemma: distancia BMOAmu}. It remains to prove
    \begin{equation}\label{eq: equiv distancia 2}
    \limsup_{|a|\to 1^-}\left ( \frac{\int_{S(a)}\abs{D^\mu(g)(z)}^2\frac{\mug(z)^2}{1-|z|}dA(z)}{1-|a|}\right )^{\frac{1}{2}} \lesssim \nm{V_{\mu,g}}_{e,H^p\to H^p} \lesssim \dist(g,\VMOA_\mu).
    \end{equation}
    The second inequality as in the proof of Proposition~\ref{propo: norm essential 1}, using \cite[Theorem 1.3]{BelMorNikPel}. We therefore prove the first inequality.
    Consider the test functions,
    $$
    f_{a,p}(z)=\frac{(1-|a|)^{2}}{(1-\overline{a}z)^{2+\frac{1}{p}}},\quad a\in\D.
    $$
    They satisfy $\nm{f_{a,p}}_{H^q} \asymp (1-|a|)^{\frac{1}{q}-\frac{1}{p}}$ for every $q>\frac{p}{2p+1}$. In particular, $\sup_{a\in\D} \nm{f_{a,p}}_{H^p}\lesssim 1$ and $f_{a,p}$ converges to zero uniformly on compact subsets as $|a|\to 1^-$. Applying Lemma~\ref{lemma: compact trick}, after a normalization independent of $a$, gives
    \begin{equation}\label{eq: hardy ess norm por limsup}
    \limsup_{|a|\to 1^-} \nm{V_{\mu,g}(f_{a,p})}_{H^p} \leq C \nm{V_{\mu,g}}_{e,H^p\to H^p}.
    \end{equation}
    Then it remains to check that 
    \begin{equation}\label{eq: final limsup}
         \limsup_{|a|\to 1^-}\left ( \frac{\int_{S(a)}\abs{D^\mu(g)(z)}^2\frac{\mug(z)^2}{1-|z|}dA(z)}{1-|a|}\right )^{\frac{1}{2}} \lesssim \limsup_{|a|\to 1^-} \nm{V_{\mu,g}(f_{a,p})}_{H^p}.
    \end{equation}

    To prove \eqref{eq: final limsup}, we treat $p<2$, $p=2$ and $p>2$ separately.

    If $p=2$, \cite[Corollary 3.4]{BelMorNikPel} gives a suitable expression for the $H^2$ norm in terms of $D^\mu(f)$. Thus, for every $a\in\D$,
    \begin{equation*}
        \begin{split}
    \nm{V_{\mu,g}(f_{a,2})}_{H^2}^2 &\asymp \int_\D \abs{f_{a,2}(z)}^2 \abs{D^\mu(g)(z)}^2 \frac{\mug(z)^2}{1-|z|}dA(z) \\
    &\gtrsim \frac{\int_{S(a)}\abs{D^\mu(g)(z)}^2 \frac{\mug(z)^2}{1-|z|}dA(z)}{1-|a|},
        \end{split}
    \end{equation*}
and \eqref{eq: final limsup} follows by taking the superior limit.

For $p>2$, Hölder's inequality and \cite[Theorem 1.3]{BelMorNikPel} yield
\begin{equation}\label{eq: caso p>2 1}
    \begin{split}
        \nm{V_{\mu,g}(f_{a,p})}_{H^2}^2 &\leq \nm{V_{\mu,g}(f_{a,p})}_{H^p}^{\frac{p}{p-1}}\nm{V_{\mu,g}(f_{a,p})}_{H^1}^{\frac{p-2}{p-1}}\\
        &\lesssim \nm{V_{\mu,g}(f_{a,p})}_{H^p}^{\frac{p}{p-1}} \nm{g}_{\BMOA_\mu}^{\frac{p-2}{p-1}}\nm{f_{a,p}}_{H^1}^{\frac{p-2}{p-1}} \\
        &\asymp \nm{V_{\mu,g}(f_{a,p})}_{H^p}^{\frac{p}{p-1}} \nm{g}_{\BMOA_\mu}^{\frac{p-2}{p-1}}(1-|a|)^{1-\frac{2}{p}}.
    \end{split}
\end{equation}
On the other hand, arguing as in the case $p=2$, we obtain, for every $a\in \D$
$$
\nm{V_{\mu,g}(f_{a,p})}_{H^2}^2 \gtrsim \frac{\int_{S(a)}\abs{D^\mu(g)(z)}^2 \frac{\mug(z)^2}{1-|z|}dA(z)}{(1-|a|)^{\frac{2}{p}}},
$$
combining this with \eqref{eq: caso p>2 1}, for every $a\in \D$,
$$
\frac{\int_{S(a)}\abs{D^\mu(g)(z)}^2 \frac{\mug(z)^2}{1-|z|}dA(z)}{1-|a|}  \lesssim \nm{V_{\mu,g}(f_{a,p})}_{H^p}^{\frac{p}{p-1}} \nm{g}_{\BMOA_\mu}^{\frac{p-2}{p-1}}.
$$
This inequality holds for every symbol in $\BMOA_\mu$. In particular, for every $h\in \VMOA_\mu$ and $a\in\D$,
\begin{equation}\label{eq: caso p>2 2}
    \frac{\int_{S(a)}\abs{D^\mu(g-h)(z)}^2 \frac{\mug(z)^2}{1-|z|}dA(z)}{1-|a|} \lesssim \nm{V_{\mu,g-h}(f_{a,p})}_{H^p}^{\frac{p}{p-1}} \nm{g-h}_{\BMOA_\mu}^{\frac{p-2}{p-1}}.
\end{equation}
Before taking the limit superior as $|a|\to 1^-$ in \eqref{eq: caso p>2 2}, we note two facts.
First, the triangle inequality and $h\in \VMOA_\mu$ give
$$
\limsup_{|a|\to 1^-}\frac{\int_{S(a)}\abs{D^\mu(g-h)(z)}^2 \frac{\mug(z)^2}{1-|z|}dA(z)}{1-|a|}  = \limsup_{|a|\to 1^-}\frac{\int_{S(a)}\abs{D^\mu(g)(z)}^2 \frac{\mug(z)^2}{1-|z|}dA(z)}{1-|a|}.
$$
Second, \eqref{eq: hardy ess norm por limsup} and the compactness of $V_{\mu,h}$ imply
$$
0\leq \limsup_{|a|\to 1^-} \nm{V_{\mu,h}(f_{a,p})}_{H^p} \lesssim \nm{V_{\mu,h}}_{e,H^p\to H^p}=0.
$$
Thus, by the triangle inequality,
$$
\limsup_{|a|\to 1^-} \nm{V_{\mu,g-h}(f_{a,p})}_{H^p} = \limsup_{|a|\to 1^-} \nm{V_{\mu,g}(f_{a,p})}_{H^p}.
$$
Taking the limit superior as $|a|\to 1^-$ in \eqref{eq: caso p>2 2} gives, for every $h\in\VMOA_\mu$
$$
\limsup_{|a|\to 1^-}\frac{\int_{S(a)}\abs{D^\mu(g)(z)}^2 \frac{\mug(z)^2}{1-|z|}dA(z)}{1-|a|}  \lesssim \limsup_{|a|\to 1^-}\nm{V_{\mu,g}(f_{a,p})}_{H^p}^{\frac{p}{p-1}} \nm{g-h}_{\BMOA_\mu}^{\frac{p-2}{p-1}}.
$$
Taking the infimum over $h\in \VMOA_\mu$ gives
$$
\limsup_{|a|\to 1^-}\frac{ \int_{S(a)}\abs{D^\mu(g)(z)}^2 \frac{\mug(z)^2}{1-|z|}dA(z)}{1-|a|} \lesssim \limsup_{|a|\to 1^-}\nm{V_{\mu,g}(f_{a,p})}_{H^p}^{\frac{p}{p-1}} \dist(g,\VMOA_\mu)^{\frac{p-2}{p-1}}.
$$
Without loss of generality, we may assume that $g\in\BMOA_\mu\setminus \VMOA_\mu$, otherwise the result is trivial. Then we can divide by the distance and apply Lemma~\ref{lemma: distancia BMOAmu} to obtain that
$$
\left (\limsup_{|a|\to 1^-}\frac{\int_{S(a)}\abs{D^\mu(g)(z)}^2 \frac{\mug(z)^2}{1-|z|}dA(z)}{1-|a|}\right )^{\frac{p}{2(p-1)}} \lesssim \limsup_{|a|\to 1^-}\nm{V_{\mu,g}(f_{a,p})}_{H^p}^{\frac{p}{p-1}},
$$
or equivalently
$$
\limsup_{|a|\to 1^-}\left (\frac{\int_{S(a)}\abs{D^\mu(g)(z)}^2 \frac{\mug(z)^2}{1-|z|}dA(z)}{1-|a|} \right )^{\frac{1}{2}} \lesssim \limsup_{|a|\to 1^-}\nm{V_{\mu,g}(f_{a,p})}_{H^p},
$$
this proves \eqref{eq: final limsup} for $p>2$.

It remains to consider $p<2$. Write $2=\frac{p}{2} + \frac{4-p}{2}$, Hölder's inequality with exponent $2$ and \cite[Theorem 1.3]{BelMorNikPel} give
\begin{equation}\label{eq: caso p<2 1}
    \begin{split}
        \nm{V_{\mu,g}(f_{a,p})}_{H^2}^2 &\leq \nm{V_{\mu,g}(f_{a,p})}_{H^p}^{\frac{p}{2}}\nm{V_{\mu,g}(f_{a,p})}_{H^{4-p}}^{\frac{4-p}{2}}\\
        &\lesssim \nm{V_{\mu,g}(f_{a,p})}_{H^p}^{\frac{p}{2}} \nm{g}_{\BMOA_\mu}^{\frac{4-p}{2}}\nm{f_{a,p}}_{H^{4-p}}^{\frac{4-p}{2}} \\
        &\asymp \nm{V_{\mu,g}(f_{a,p})}_{H^p}^{\frac{p}{2}}\nm{g}_{\BMOA_\mu}^{\frac{4-p}{2}}(1-|a|)^{1-\frac{2}{p}},
    \end{split}
\end{equation}
Arguing as in the case $p>2$, we obtain, for every $a\in \D$,
$$
\frac{\int_{S(a)}\abs{D^\mu(g)(z)}^2 \frac{\mug(z)^2}{1-|z|}dA(z)}{1-|a|}  \lesssim \nm{V_{\mu,g}(f_{a,p})}_{H^p}^{\frac{p}{2}}\nm{g}_{\BMOA_\mu}^{\frac{4-p}{2}},
$$
and again applying this estimate to $g-h$ for every $h\in \VMOA_\mu$ and taking the infimum
$$
\limsup_{|a|\to 1^-}\frac{ \int_{S(a)}\abs{D^\mu(g)(z)}^2 \frac{\mug(z)^2}{1-|z|}dA(z)}{1-|a|} \lesssim \limsup_{|a|\to 1^-}\nm{V_{\mu,g}(f_{a,p})}_{H^p}^{\frac{p}{2}} \dist(g,\VMOA_\mu)^{\frac{4-p}{2}},
$$
Now, we may assume that $g\in \BMOA_\mu\setminus\VMOA_\mu$. Dividing by the appropriate power of the distance and applying Lemma~\ref{lemma: distancia BMOAmu} gives
$$
\left (\limsup_{|a|\to 1^-}\frac{\int_{S(a)}\abs{D^\mu(g)(z)}^2 \frac{\mug(z)^2}{1-|z|}dA(z)}{1-|a|}\right )^{\frac{p}{4}} \lesssim \limsup_{|a|\to 1^-}\nm{V_{\mu,g}(f_{a,p})}_{H^p}^{\frac{p}{2}},
$$
so \eqref{eq: final limsup} holds for $p<2$, concluding the proof.
\end{proof}

We conclude the section by proving Theorem~\ref{th: norma esencial p=q}.

\begin{Prf}{\em{  Theorem~\ref{th: norma esencial p=q}}.}
Statement $(i)$ follows by combining Proposition~\ref{propo: norm essential 2}, \cite[Theorem 1.2]{BelMorNikPel} and Lemma~\ref{lemma: distancia Bbn} (ii):
\begin{equation*}
    \begin{split}
        \nm{V_{\mu,g}}_{e,H^p\to H^p} &\asymp \dist(g,\VMOA_\mu) = \inf_{h\in\VMOA_\mu}\nm{g-h}_{\BMOA_\mu} \asymp \inf_{h\in\VMOA}\nm{g-h}_{\BMOA} \\ &=  \dist(g,\VMOA) \asymp  \limsup_{|a|\to 1^-}\left ( \frac{\int_{S(a)}\abs{g'(z)}^2 (1-|z|)dA(z)}{1-|a|}\right )^{\frac{1}{2}}.
    \end{split}
\end{equation*}
Similarly, Proposition~\ref{propo: norm essential 1} $(ii)$ with $p=q$, Theorem~\ref{th: caract Bmu} with $\b=0$, and Lemma~\ref{lemma: distancia Bbn} (i) give
\begin{equation*}
\begin{split}
\nm{V_{\mu,g}}_{e,A^p_\a\to A^p_\a} &\asymp \dist(g,\B^{0,\mu}_0) = \inf_{h\in\B^{0,\mu}_0}\nm{g-h}_{\B^{0,\mu}} \asymp \inf_{h\in\B_0}\nm{g-h}_{\B} \\
&= \dist(g,\B_0) \asymp \limsup_{|z|\to 1^-}\abs{g'(z)}(1-|z|),
\end{split}
\end{equation*}
which proves $(ii)$.
\end{Prf}
\bibliographystyle{amsplain}
\bibliography{Bibliografia}  
\end{document}